\documentclass[letterpaper,12pt]{article}

\usepackage[margin=0.5in]{geometry}

\usepackage{fullpage}
\usepackage{enumerate}
\usepackage{authblk}
\usepackage[colorinlistoftodos]{todonotes}
\usepackage{verbatim}
\usepackage{section, amsthm, textcase, setspace, amssymb, lineno, amsmath, amssymb, amsfonts, latexsym, fancyhdr, longtable, ulem, mathtools}
\usepackage{epsfig, graphicx, pstricks,pst-grad,pst-text,tikz, tkz-berge,colortbl}
\usepackage{epsf}
\usepackage{graphicx, color}
\usepackage{float}
\usepackage[rflt]{floatflt}
\usepackage{amsfonts}
\usepackage{latexsym,enumitem}
\usetikzlibrary{fit,matrix,positioning}
\usepackage{pdflscape}
\usepackage{lineno}
\usepackage{kbordermatrix}

\usepackage{mathrsfs}

\usetikzlibrary{decorations.pathreplacing}

\definecolor{darkgreen}{rgb}{0, 0.5, 0}

\newtheorem{theorem}{Theorem}
\newtheorem{lemma}{Lemma}
\newtheorem{corollary}{Corollary}

\newtheorem{definition}{Definition}

\newtheorem{Ex}{Example}

\newtheorem*{theorem*}{Theorem}
\newtheorem{remark}{Remark}

\newcommand{\rn}[1]{{\color{red} #1}}

\newcommand\addvmargin[1]{
  \node[fit=(current bounding box),inner ysep=#1,inner xsep=0]{};}
  
\newcommand{\ind}{{\rm ind \hspace{.1cm}}}

\newcommand\scalemath[2]{\scalebox{#1}{\mbox{\ensuremath{\displaystyle #2}}}}

\newcommand\blfootnote[1]{%
  \begingroup
  \renewcommand\thefootnote{}\footnote{#1}%
  \addtocounter{footnote}{-1}%
  \endgroup
}

\begin{document}

\title{The index of Lie poset algebras }

\author[*]{Vincent E. Coll, Jr.}
\author[*]{Nick W. Mayers}

\affil[*]{Department of Mathematics, Lehigh University, Bethlehem, PA, 18015}

\maketitle

\begin{center}
\textit{Dedicated to Murray Gerstenhaber - teacher and mentor}
\end{center}
\begin{abstract}
\noindent
We provide general closed-form formulas for the index of type-A Lie poset algebras 
corresponding to posets of restricted height.  Furthermore, we provide a combinatorial recipe for constructing all posets corresponding to type-A Frobenius Lie poset algebras of heights zero, one, and two. A finite Morse theory argument establishes that the simplicial realization of such posets is contractible.  It then follows, from a recent theorem of Coll and Gerstenhaber, that the second Lie cohomology group of the corresponding Lie poset algebra with coefficients in itself is zero.   Consequently, such a Lie poset algebra is absolutely rigid and cannot be deformed. We also provide matrix representations for Lie poset algebras in the other classical types. By so doing, we are able to give examples of deformable Lie algebras which are both solvable and Frobenius. This resolves a question of Gerstenhaber and Giaquinto about the existence of such algebras.
\end{abstract}
\blfootnote{\textit{E-mail addresses:} vec208@lehigh.edu (V. Coll), nwm215@lehigh.edu (N. Mayers).}



\section{Introduction}

\bigskip  The incidence algebra $I(\mathcal{P},\textbf{k})$ of a finite poset ($\mathcal{P}, \preceq)$ over a field $\textbf{k}$ is an associative $\textbf{k}$-algebra consisting of all functions $f:Int(\mathcal{P})\to \textbf{k}$ mapping closed intervals of $\mathcal{P}$ to $\textbf{k}$ with multiplication given by the convolution product

$$
(f*g)([x, y])=\sum_{x\preceq z\prec y}f([x, z])g([z, y]).
$$

\noindent
By taking a linear extension of $\mathcal{P}$, one can represent $I(\mathcal{P},\textbf{k})$ as a matrix algebra of $|\mathcal{P}|\times|\mathcal{P}|$ matrices. For $f\in I(\mathcal{P},\textbf{k})$, the matrix representing $f$ has $i,j$-entry equal to $f([i,j])$ and, with this choice of representation, the convolution product between elements of $I(\mathcal{P},\textbf{k})$ becomes matrix multiplication. Incidence algebras were introduced into combinatorics in 1964 by Rota as a means of studying inversion-type formulas in a unified way \textbf{\cite{Ro}}. Since then, these algebras have been studied by many authors, from many different perspectives  -- and have been regularly rediscovered and called, variously, $T^3$ algebras, convolution rings of posets, incidence matrix rings of posets, pattern algebras, tic-tac-toe algebras, and poset algebras \textbf{\cite{G4,Lar1,Mi}}. We prefer the name \textit{poset algebras} since posets generate them in much the same way as groups generate group algebras.  

Poset algebras can be naturally endowed with a Lie structure by taking the commutator product. Even so, the study of such Lie poset algebras has only recently been initiated (\textbf{\cite{CG}}, 2016).  In \textbf{\cite{CG}}, Coll and Gerstenhaber define \textit{Lie poset algebras} as those subalgebras of the classical Lie algebras which lie between a Cartan subalgebra and a Borel subalgebra; and go on to compute, in particular, the Lie algebra cohomology of a Lie poset algebra with coefficients in itself. This is the controlling cohomology for the infinitesimal deformations of the Lie poset algebra \textbf{\cite{TG}}. Generally, Lie poset algebras deform  -- even if the underlying associative poset algebra does not -- although the deformed Lie algebra may no longer be a Lie poset algebra.   See \textbf{\cite{CG}} for examples. 

In this article, we are interested in the interaction between the deformation theory of a Lie poset algebra and its index -- with a special emphasis on Lie poset subalgebras of the first classical type;  that is, type $A_{n-1}=\mathfrak{sl}(n)$.  The \textit{index} of a general Lie algebra $\mathfrak{g}$ is an important algebraic invariant introduced by Dixmier (\textbf{\cite{D}}, 1974) and is defined as follows: 

\[\ind \mathfrak{g}=\min_{F\in \mathfrak{g^*}} \dim  (\ker (B_F)),\]

\noindent where $B_F$ is the skew-symmetric \textit{Kirillov form} defined by $B_F(x,y)=F([x,y])$ for all $x,y\in\mathfrak{g}$.

Combinatorial methods for the computation of a Lie algebra's index are
of great topical interest \textbf{\cite{Cameron,Coll3,Coll1,CHM,Coll2,DK,Elash,Joseph,Panyushev1,Panyushev2,Panyushev3}}.  In this article, we initiate the study of the index theory for Lie poset algebras, paying special attention to those algebras which have index zero.  Index-zero Lie algebras are called  \textit{Frobenius} and are of particular interest in deformation theory.\footnote{
Suppose $B_F(-,-)$ is non-degenerate on $\mathfrak{g}$ and let $[F]$ be the matrix of $B_F(-,-)$ relative to some basis $\{x_1,\dots,x_n  \}$ of $\mathfrak{g}$; such a functional $F\in\mathfrak{g}^*$ is referred to as \textit{Frobenius}. In \textbf{\cite{BD}}, Belavin and Drinfel'd showed that  $\sum_{i,j}[F]^{-1}_{ij}x_i\wedge x_j$ is the infinitesimal of a \textit{Universal Deformation Formula}. Such formulas can be used to deform the universal enveloping algebra of $\mathfrak{g}$ or the function space on any Lie group which contains $\mathfrak{g}$ in its Lie algebra of derivations.  Thus, each pair consisting of a Frobenius Lie algebra $\mathfrak{g}$ together with a Frobenius functional $F$ provides a constant solution to the classical Yang-Baxter equation (see (\textbf{\cite{G1}}, 1997) and (\textbf{\cite{G2}}, 2008)).} 


We also introduce definitions of posets of types B, C, and D. As with type A, these poset definitions are used to develop natural matrix representations.  Subsequently, we produce an example of a non-rigid Frobenius Lie poset algebra (in types B, C, and D). This resolves a question of Gerstenhaber and Giaquinto about whether a deformable Frobenius Lie algebra could exist \textbf{\cite{Prin}}.
Finally, in an extended epilogue, we examine some other motivations for our study. This includes suggestive results regarding the recently developed spectral theory of Frobenius Lie algebra (see \textbf{\cite{unbroken, specAB, Cameron})}.

\bigskip 
There are two main results in this paper.  The first is the development of closed-form index formulas for type-A Lie poset algebras corresponding to posets of heights zero, one, and two. The non-trivial height-two case is treated in Theorem~\ref{thm:gind}, where the attendant index formula given in equation (\ref{eqn:grind}) may be regarded as the main combinatorial result of this paper.\footnote{This result is a central result in the second author's Ph.D. thesis (in progress) ``The index of Lie poset algebras" at Lehigh University \textbf{\cite{Mayers}}.}  The index formula of Theorem~\ref{thm:gind} subsequently yields a characterization of posets of heights one, and two, which are associated to Frobenius Lie poset algebras.  In height two, this characterization takes the form of a combinatorial recipe -- building blocks and gluing rules -- for the construction of all such posets (see Theorem~\ref{thm:h2frobchar}). A discrete Morse theory argument then establishes that the simplicial complex associated with any such poset $\mathcal{P}$ is contractible, so has no simplicial homology (see Theorem~\ref{Nohomology}). A recent result of Coll and Gerstenhaber (Theorem~\ref{CG}) can then be applied to find that the second Lie cohomology group of the corresponding type-A Lie poset algebra with coefficients in  
itself is zero.\footnote{Theorem~\ref{CG} is the Lie algebraic analogue of the now classical result of Gerstenhaber and Schack which asserts that simplicial cohomology is a special case of  Hochschild cohomology \textbf{\cite{G3}}.  }

Putting all of this together yields the second main result of this paper (see Theorem \ref{thm:main2}). 
\begin{theorem*}

A Frobenius Lie poset subalgebra of $\mathfrak{sl}(n)$ corresponding to a poset of height zero, one, or two is absolutely rigid.
\end{theorem*}

\begin{remark}  Extensive simulations suggest that the above theorem is true for all heights.  See the commentary at the end of Section~\ref{sec:rigid}.  Also note that since all Lie poset algebras are solvable, this theorem, in conjunction with Theorems~\ref{thm:h01FC} and \ref{thm:npurechar}, allows for the production of non-trivial examples of solvable Lie algebras which are absolutely rigid.
\end{remark}

\section{Lie poset algebras}
In this section, and following \textbf{\cite{CG}}, we provide a general definition of Lie poset algebras applicable to any Chevalley-type Lie algebra.  We specialize to Lie poset subalgebras of $\mathfrak{sl}(n)$ by providing explicit matrix representations.  Throughout, assume that \textbf{k} is an algebraically closed field of characteristic zero, which we may take to be the complex numbers.

\bigskip
Let $(\mathcal{P}, \preceq_{\mathcal{P}})$ be a finite poset with partial order $\preceq_{\mathcal{P}}$. (It will cause no confusion to simply write $\mathcal{P}$ when the partial ordering is understood and we will usually suppress the subscript in $\preceq_{\mathcal{P}}$.) The \textit{associative poset algebra} $A(\mathcal{P})=A(\mathcal{P}, \textbf{k})$ is the span over $\textbf{k}$ of elements $e_{ij}$, $i\preceq j$ with multiplication given by setting $e_{ij}e_{kl}=e_{il}$ if $j=k$, and $0$ otherwise.  The \textit{trace} of an element $\sum c_{ij}e_{ij}$ is $\sum c_{ii}.$

We can equip $A(\mathcal{P})$ with the commutator product $[a,b]=ab-ba$, where concatenation denotes the product in $A(\mathcal{P})$, to produce the \textit{Lie poset algebra} $\mathfrak{g}(\mathcal{P})=\mathfrak{g}(\mathcal{P}, \textbf{k})$. If $|\mathcal{P}|=n$, then it is possible to represent $\mathcal{P}$ as a poset on $\{1,\hdots,n\}$, where $\preceq_{\mathcal{P}}$ is compatible with the linear ordering, via an order-preserving bijection. Such a representation of $\mathcal{P}$ is called a \textit{linear extension} of $\mathcal{P}$. Taking a linear extension of $\mathcal{P}$, the associative algebra $A(\mathcal{P})$ and the Lie algebra $\mathfrak{g}(\mathcal{P})$ may be regarded, respectively, as associative and Lie subalgebras of the algebra of all upper triangular $n \times n$ matrices over $\textbf{k}$. Such a matrix representation is realized by replacing each basis element $e_{i,j}$ by the $n\times n$ matrix $E_{i,j}$ containing a 1 in the $i,j$-entry, and 0's elsewhere. The product between elements $e_{i,j}$ is then replaced by matrix multiplication between the $E_{i,j}$; it is well-known that such matrix algebras are invariant under the choice of linear extension of $\mathcal{P}$. Let $\mathfrak{b}$ be the Borel subalgebra of $n\times n$ matrices consisting of upper triangular matrices of trace zero and $\mathfrak{h}$ its Cartan subalgebra of diagonal matrices. Any subalgebra $\mathfrak{g}$ lying between $\mathfrak{h}$ and $\mathfrak{b}$ is then a Lie poset algebra; for $\mathfrak{g}$ is then the span over $\textbf{k}$ of $\mathfrak{h}$ and those $E_{i,j}$ which it contains, and there is a partial order on $\mathcal{P}=\{1,\dots,n\}$ compatible with the linear order by setting $i\preceq j$ whenever $E_{i,j}\in \mathfrak{g}$ (see Example \ref{Stargate}). Restricting $\mathfrak{g}(\mathcal{P})$ to 
trace zero matrices yields a subalgebra of 
the first classical family $A_{n-1}=\mathfrak{sl}(n)$.  We denote the resulting type-A Lie poset algebra by $\mathfrak{g}_A(\mathcal{P})$.  Since $A(\mathcal{P},\textbf{k})$ is isomorphic to $I(\mathcal{P},\textbf{k})$, note that the definition of $\mathfrak{g}_A(\mathcal{P})$ is consistent with the definition of a type-A Lie poset algebra given in the Introduction,

\begin{Ex}\label{Stargate}  
The Hasse diagram of a poset is a graphical representation of the poset with an implied upward orientation.  Consider the poset $\mathcal{P} =$ $\{1,2,3,4\}$ with $1 \preceq 2 \preceq 3,4$ and no relations other than those following from these. The Hasse diagram of $\mathcal{P}$ is illustrated in Figure~\ref{First} \textup(left\textup). The basic form of the matrix algebra $A(\mathcal{P})$, respectively $\mathfrak{g}(\mathcal{P})$, is illustrated in Figure~\ref{First} \textup(right\textup); the *'s indicate possible non-zero entries from $\mathbb{C}$.

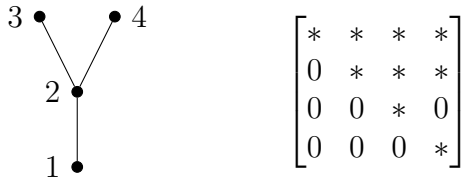
\begin{figure}[H]
$$\begin{tikzpicture}
	\node (1) at (0, 0) [circle, draw = black, fill = black, inner sep = 0.5mm, label=left:{1}]{};
	\node (2) at (0, 1)[circle, draw = black, fill = black, inner sep = 0.5mm, label=left:{2}] {};
	\node (3) at (-0.5, 2) [circle, draw = black, fill = black, inner sep = 0.5mm, label=left:{3}] {};
	\node (4) at (0.5, 2) [circle, draw = black, fill = black, inner sep = 0.5mm, label=right:{4}] {};
	\node (5) at (4,1) {$\begin{bmatrix}
   * & * & * & *  \\
   0 & * & * & * \\
   0 & 0 & * & 0 \\
   0 & 0 & 0 & *
\end{bmatrix}$};
    \draw (1)--(2);
    \draw (2)--(3);
    \draw (2)--(4);
    \addvmargin{1mm}
\end{tikzpicture}$$
\caption{Hasse diagram of $\mathcal{P}$ (left) and associated matrix algebra (right)}\label{First}
\end{figure}
\end{Ex}

\noindent
We continue to set the combinatorial notation.

Let $Rel(\mathcal{P})$ denote the set of strict relations between elements of $\mathcal{P}$, $Ext(\mathcal{P})$ denote the set of minimal and maximal elements of $\mathcal{P}$, and $Rel_E(\mathcal{P})$ denote the number of strict relations between the elements of $Ext(\mathcal{P})$.

\begin{Ex}
Let $\mathcal{P}$ be the poset in 
Example~\ref{Stargate}.  We have $$Rel(\mathcal{P})=\{1\prec 2,1\preceq 3,1\preceq 4,2\preceq 3,2\preceq 4\},$$  $$Ext(\mathcal{P})=\{1,3,4\},\quad \text{and}\quad Rel_E(\mathcal{P})=\{1\preceq 3, 1\preceq 4\}.$$
\end{Ex}

\noindent
Recall that, if $x\preceq y$ and there exists no $z\in \mathcal{P}$ satisfying $x,y\neq z$ and $x\preceq z\preceq y$, then $y$ \textit{covers} $x$ and 
$x\preceq y$ is a \textit{covering relation}.  Using this language, the Hasse diagram of a poset $\mathcal{P}$ can be reckoned as the graph whose vertices correspond to elements of $\mathcal{P}$ and whose edges correspond to covering relations. A poset $\mathcal{P}$ is \textit{connected} if the Hasse diagram of $\mathcal{P}$ is connected as a graph. Throughout this paper, $C_{\mathcal{P}}$ will denote the number of connected components of the Hasse diagram of $\mathcal{P}$.

Given a subset $S\subset\mathcal{P}$, the \textit{induced subposet generated by $S$} is the poset $\mathcal{P}_S$ on $S$, where $i\preceq_{\mathcal{P}_S}j$ if and only if $i\preceq_{\mathcal{P}}j$. A totally ordered subset $S\subset\mathcal{P}$ is called a \textit{chain}. Using the chains of a poset $\mathcal{P}$ one can define a simplicial complex $\Sigma(P)$, where the vertices represent the elements of $\mathcal{P}$ and the faces are chains. 

\begin{Ex}
Let $\mathcal{P}$ be the poset of our running Example~\ref{Stargate}.  The simplicial complex $\Sigma(\mathcal{P})$ is illustrated in Figure \ref{fig:simpcomp} below.

\begin{figure}[H]
$$\begin{tikzpicture}
	\node (1) at (0, 0) [circle, draw = black, fill = black, inner sep = 0.5mm, label=below:{1}]{};
	\node (2) at (-1, 0.5)[circle, draw = black, fill = black, inner sep = 0.5mm,label=left:{3}] {};
    \node (3) at (0, 1) [circle, draw = black, fill = black, inner sep = 0.5mm,label=above:{2}]{};
    \node (4) at (1, 0.5) [circle, draw = black, fill = black, inner sep = 0.5mm,label=right:{4}]{};
    \draw [fill=gray,thick] (0, 0)--(-1, 0.5)--(0, 1)--(0, 0);
    \draw [fill=gray,thick] (0, 0)--(1, 0.5)--(0, 1)--(0, 0);
\end{tikzpicture}$$
\caption{Simplicial complex $\Sigma(\mathcal{P})$}\label{fig:simpcomp}
\end{figure}
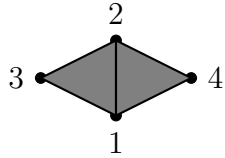
\end{Ex}

\noindent
A chain $S\subset\mathcal{P}$ is called \textit{maximal} if it is not a proper subset of any other chain $S'\subset \mathcal{P}$. If every maximal chain of a poset $\mathcal{P}$ is of the same length, then we call $\mathcal{P}$ \textit{pure}. When a poset is pure, there is a natural grading on the elements of $\mathcal{P}$. This grading is made precise by a rank function $r:\mathcal{P}\to \mathbb{Z}_{\ge 0}$, where minimal elements have rank zero and if $x$ is covered by $y$ in $\mathcal{P}$, then $r(y)=r(x)+1$. Note that the poset of Example 1 is pure since it's maximal chains $1\preceq 2\preceq 3$ and $1\preceq 2\preceq 4$, both have length two; furthermore, this poset has a single minimal element of rank zero, namely $\{1\}$, a single element of rank one, namely $\{2\}$, and two maximal elements of rank two, namely $\{3,4\}$. We define the \textit{height} of a poset $\mathcal{P}$ to be one less than the cardinality of the largest chain in $\mathcal{P}$. Note that when a poset $\mathcal{P}$ is pure, its height is equal to $\max_{x\in \mathcal{P}}r(x)$. 

We end this section with the definition of a family of posets, as well as two poset operations which will be important in the sections that follow.

\begin{definition}\label{def:comppos}
Let $\mathcal{P}$ be the poset with $r_i$ elements of rank $i$, for $0\le i\le t$, and every possible relation between elements of differing rank. We denote such ``complete" posets by $\mathcal{P}(r_0, r_1,\hdots, r_t)$.  See Figure~\ref{fig:non}.
\end{definition}

\begin{figure}[H]
$$\begin{tikzpicture}
	\node (1) at (0, 1) [circle, draw = black, fill = black, inner sep = 0.5mm]{};
	\node (2) at (-1, 0)[circle, draw = black, fill = black, inner sep = 0.5mm] {};
    \node (3) at (0, 0) {$\cdots$};
	\node (4) at (1, 0) [circle, draw = black, fill = black, inner sep = 0.5mm] {};
	\node (5) at (0,-1) {$\mathcal{P}(n,1)$};
    \draw (1)--(2);
    \draw (1)--(3);
    \draw (1)--(4);
\end{tikzpicture}\quad\quad\begin{tikzpicture}
	\node (1) at (0, 0) [circle, draw = black, fill = black, inner sep = 0.5mm]{};
	\node (2) at (-1, 1)[circle, draw = black, fill = black, inner sep = 0.5mm] {};
    \node (3) at (0, 1) {$\cdots$};
	\node (4) at (1, 1) [circle, draw = black, fill = black, inner sep = 0.5mm] {}; \node (5) at (0,-1) {$\mathcal{P}(1,n)$};   
    \draw (1)--(2);
    \draw (1)--(3);
    \draw (1)--(4);
\end{tikzpicture}$$ 
\caption{Complete posets of height one}\label{fig:non}
\end{figure}
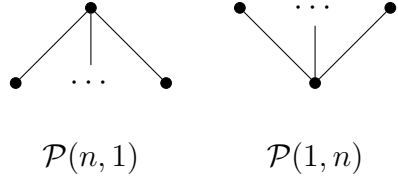

\begin{Ex}
Using the notation of Definition~\ref{def:comppos}, the poset of Example~\ref{Stargate} is $\mathcal{P}(1,1,2)$.
\end{Ex}

\begin{definition}
Given two posets $\mathcal{P}$ and $\mathcal{Q}$ which are disjoint as sets, the \textit{disjoint union} of $\mathcal{P}$ and $\mathcal{Q}$ is the poset $\mathcal{P}+\mathcal{Q}$ on the union $\mathcal{P}\cup \mathcal{Q}$ such that $s\le t$ in $\mathcal{P}+\mathcal{Q}$ if either 
\begin{enumerate}[label=\textup(\roman*\textup)]
    \item $s,t\in \mathcal{P}$ and $s\preceq_{\mathcal{P}} t$, or
    \item $s,t\in \mathcal{Q}$ and $s\preceq_{\mathcal{Q}} t$.
\end{enumerate}
Note: When there are more than two posets 
$\mathcal{P}_i$, for $1\le i\le n$, we will denote their disjoint union as $\sum_{i=1}^n\mathcal{P}_i$.
\end{definition}

\begin{definition}
If $\mathcal{P}$ is a poset, define its dual poset $\mathcal{P}^*$ by the following rules

\begin{enumerate}[label=\textup(\roman*\textup)]
    \item $i,j\in \mathcal{P}^*$ if $i,j\in \mathcal{P}$ and
    \item $j\preceq_{\mathcal{P}^*}i$ if $i\preceq_{\mathcal{P}}j$.
\end{enumerate}
\end{definition}

\section{Combinatorial index formulas}

In this section, we develop general closed-form formulas for the index of type-A Lie poset algebras corresponding to posets of height zero (Theorem \ref{h0ind}) as well as heights one and two (Theorem \ref{thm:gind}).  

It will be convenient to use an 
alternative characterization of the index. Let $\mathfrak{g}$ be an arbitrary Lie algebra with basis $\{x_1,...,x_n\}$.  The index of $\mathfrak{g}$ can be expressed using the \textit{commutator matrix}, $C(\mathfrak{g})=([x_i,x_j])_{1\le i<j\le n}$, over the quotient field $R(\mathfrak{g})$ of the symmetric algebra $Sym(\mathfrak{g})$ as follows (see \textbf{\cite{D}}).

\begin{theorem}\label{thm:commat}
 The index of $\mathfrak{g}$ is given by  
 
 $$\ind \mathfrak{g}= n-Rank_{R(\mathfrak{g})}C(\mathfrak{g}).$$ 
\end{theorem}

\begin{Ex}
Consider $\mathfrak{g}_A(\mathcal{P}(1,1))$; that is, 
the collection of upper triangular matrices in $\mathfrak{sl}(2)$. A Chevalley basis for $\mathfrak{g}_A(\mathcal{P}(1,1))$ is given by $\{x_1,x_2\}$, where $[x_1,x_2]=2x_2$. The standard  matrix representation of $\mathfrak{g}_A(\mathcal{P}(1,1))$ is illustrated in Figure~\ref{ex:commatsl2}.  Since the rank of this matrix is two, it follows from Theorem~\ref{thm:commat} that $\mathfrak{g}_A(\mathcal{P}(1,1))$ is Frobenius.
\begin{figure}[H]
$$\begin{bmatrix}
   0 & 2x_2  \\
     -2x_2 & 0 
\end{bmatrix}$$
\caption{$C(\mathfrak{g}_A(\mathcal{P}(1,1)))$}\label{ex:commatsl2}
\end{figure}
\end{Ex}

\subsection{A matrix reduction}\label{sec:ifg}
Assume for the moment that $\mathcal{P}$ is connected.
To better understand $C(\mathfrak{g}(\mathcal{P}))$, we develop a basis 
for $\mathfrak{g}(\mathcal{P})$, where, following a certain reduction algorithm, $C(\mathfrak{g}(\mathcal{P}))$ nicely reduces to the equivalent matrix $C'(\mathfrak{g}(\mathcal{P}))$.  In $C'(\mathfrak{g}(\mathcal{P}))$ there are $Rel_E(\mathcal{P})-Ext(\mathcal{P})+1$ zero rows on top,  $I_k$ is the $k\times k$ identity matrix with $k=|\mathcal{P}|+|Ext(\mathcal{P})|-2$, and $B_i$, for $i=1,2,$ and $3$, are certain block matrices whose rank contributions are computable.  See Figure \ref{fig:matform}.

\begin{figure}[H]\label{reduced}
$$\begin{tikzpicture}
    \node at (2,0) {$\begin{bmatrix}
   0 & 0 & 0 & 0  \\
   $\vdots$ & $\vdots$ & $\vdots$ & $\vdots$  \\
   0 & 0 & 0 & 0  \\
   0 & I_k & B_1 & B_2 \\
   0 & 0 & 0 & B_3 \\ 
\end{bmatrix}$};
  \node at (-0.75,0) {$C'(\mathfrak{g}(\mathcal{P}))=$};
\end{tikzpicture}$$
\caption{Matrix form of $C'(\mathfrak{g}(\mathcal{P}))$}\label{fig:matform}
\end{figure}

All of the commutator matrix calculations will be done in $\mathfrak{gl}(n)$, and will be facilitated by a basis for $\mathfrak{g}(\mathcal{P})$ defined as follows. Let $E_{i,j}$ denote the $n\times n$ matrix with a 1 in the $i,j$ position, and 0's elsewhere.  Now, define the basis
$$\mathscr{B}=\bigg\{\sum_{i\in\mathcal{P}} E_{i,i}\bigg\}\cup\{E_{i,i}~|~1<i\in \mathcal{P}\}\cup\{E_{i,j}~|~i,j\in \mathcal{P}\text{ such that }i\preceq j\}.$$
Note that the row corresponding to $\sum E_{i,i}$ in $C(\mathfrak{g}(\mathcal{P}))$ is a zero row and therefore contributes $+1$ to the index of $\mathfrak{g}(\mathcal{P})$. Ignoring the contribution of this row to the index results in the index upon restriction to $\mathfrak{sl}(n)$\footnote{To see this, for a poset $\mathcal{P}$ consider the basis for $\mathfrak{g}(\mathcal{P})$ given by $\mathscr{B}$ with the basis elements $E_{i,i}$, for $i\in\mathcal{P}$, replaced by $E_{j,j}-E_{n,n}$, for $n\in\mathcal{P}$ and $n\neq j\in\mathcal{P}$. Removing the basis element $\sum_{i\in\mathcal{P}}E_{i,i}$ results in a basis for $\mathfrak{g}_A(\mathcal{P})$. Furthermore, the commutator matrices with respect to these two matrices differ only by an extra zero row and column in the commutator matrix corresponding to $\mathfrak{g}(\mathcal{P})$.}.

To ease notation, row and column labels will be bolded and matrix entries (elements of $R(\mathfrak{g}(\mathcal{P}))$) will be unbolded. Furthermore, we will refer to the row corresponding to $\mathbf{E_{i,j}}$ in $C(\mathfrak{g}(\mathcal{P}))$ -- and by a slight abuse of notation, in any equivalent matrix -- as row $\mathbf{E_{i,j}}$.

We relegate the formal steps in the general matrix reduction which transforms $C(\mathfrak{g}_A(\mathcal{P}))$ into $C'(\mathfrak{g}_A(\mathcal{P}))$ to Appendix A.

\begin{Ex}  Using the set basis $\mathscr{B}$ defined above, 
the commutator matrix of our running example $\mathfrak{g}(\mathcal{P}(1,1,2))$ is illustrated in Figure 4.   Note that without further reduction, the rank of $C(\mathfrak{g}(\mathcal{P}(1,1,2))$ is unclear. 
\begin{figure}[H]
\[
  \kbordermatrix{
    & \sum \mathbf{E_{i,i}} & \mathbf{E_{2,2}} & \mathbf{E_{3,3}} & \mathbf{E_{4,4}} & \mathbf{E_{1,2}} & \mathbf{E_{1,3}} & \mathbf{E_{1,4}} & \mathbf{E_{2,3}} & \mathbf{E_{2,4}} \\
    \sum \mathbf{E_{i,i}} & 0 & 0 & 0 & 0 & 0 & 0 & 0 & 0 & 0 \\
    \mathbf{E_{2,2}} & 0 & 0 & 0 & 0 & -E_{1,2} & 0 & 0 & E_{2,3} & E_{2,4} \\
    \mathbf{E_{3,3}} & 0 & 0 & 0 & 0 & 0 & -E_{1,3} & 0 & -E_{2,3} & 0 \\
    \mathbf{E_{4,4}} & 0 & 0 & 0 & 0 & 0 & 0 & -E_{1,4} & 0 & -E_{2,4} \\
    \mathbf{E_{1,2}} & 0 & E_{1,2} & 0 & 0 & 0 & 0 & 0 & E_{1,3} & E_{1,4} \\
    \mathbf{E_{1,3}} & 0 & 0 & E_{1,3} & 0 & 0 & 0 & 0 & 0 & 0 \\
    \mathbf{E_{1,4}} & 0 & 0 & 0 & E_{1,4} & 0 & 0 & 0 & 0 & 0 \\
    \mathbf{E_{2,3}} & 0 & -E_{2,3} & E_{2,3} & 0 & -E_{1,3} & 0 & 0 & 0 & 0 \\
    \mathbf{E_{2,4}} & 0 & -E_{2,4} & 0 & E_{2,4} & -E_{1,4} & 0 & 0 & 0 & 0 \\
  }
\]
\caption{$C(\mathfrak{g}(\mathcal{P}(1,1,2))$}\label{fig:commatrix}
\end{figure}
\noindent
Below we illustrate $C'(\mathfrak{g}(\mathcal{P}(1,1,2)))$. Note that in this case there is no $B_1$ block, since such a block only exists when $\mathcal{P}$ has more than one minimal element. Further, note that the computation of the index of $\mathfrak{g}(\mathcal{P}(1,1,2))$ has been reduced to understanding the rank of the red block in $C'(\mathfrak{g}(\mathcal{P}(1,1,2)))$. See Figure~\ref{fig:commatrixrr}.
\begin{figure}[H]
\[
  \kbordermatrix{
    & \sum \mathbf{E_{i,i}} &  & \mathbf{E_{2,2}} & \mathbf{E_{3,3}} & \mathbf{E_{4,4}} & \mathbf{E_{1,3}} & \mathbf{E_{1,4}} &  & \mathbf{E_{1,2}} & \mathbf{E_{2,3}} & \mathbf{E_{2,4}} \\
    \sum \mathbf{E_{i,i}} & 0 & \vrule  & 0 & 0 & 0 & 0 & 0 & \vrule & 0 & 0 & 0 \\ \cline{2-12}
    \mathbf{E_{1,2}} & 0 & \vrule & 1 & 0 & 0 & 0 & 0 & \vrule & 0 & \frac{E_{1,3}}{E_{1,2}} & \frac{E_{1,4}}{E_{1,2}} \\
    \mathbf{E_{1,3}} & 0 & \vrule  & 0 & 1 & 0 & 0 & 0 & \vrule & 0 & 0 & 0 \\
    \mathbf{E_{1,4}} & 0 & \vrule  & 0 & 0 & 1 & 0 & 0 & \vrule & 0 & 0 & 0 \\
    \mathbf{E_{3,3}} & 0 & \vrule  & 0 & 0 & 0 & 1 & 0 & \vrule & 0 & \frac{E_{2,3}}{E_{1,3}} & 0 \\
    \mathbf{E_{4,4}} & 0 & \vrule  & 0 & 0 & 0 & 0 & 1 & \vrule & 0 & 0 & \frac{E_{2,4}}{E_{1,4}} \\ \cline{2-12}
    \mathbf{E_{2,2}} & 0 & \vrule &  0 & 0 & 0 & 0 & 0 & \vrule & \rn{-E_{1,2}} & \rn{E_{2,3}} & \rn{E_{2,4}} \\
    \mathbf{E_{2,3}} & 0 & \vrule & 0 & 0 & 0 & 0 & 0 & \vrule & \rn{-E_{1,3}} & \rn{\frac{E_{2,3}}{E_{1,2}}E_{1,3}} & \rn{\frac{E_{2,3}}{E_{1,2}}E_{1,4}} \\
    \mathbf{E_{2,4}} & 0 & \vrule & 0 & 0 & 0 & 0 & 0 & \vrule & \rn{-E_{1,4}} & \rn{\frac{E_{2,4}}{E_{1,2}}E_{1,3}} & \rn{\frac{E_{2,4}}{E_{1,2}}E_{1,4}} \\
  }
\]
\caption{$C'(\mathfrak{g}(\mathcal{P}(1,1,2))$}\label{fig:commatrixrr}
\end{figure}
\end{Ex}

\begin{remark}\label{rem:contminmax}
Using $C'(\mathfrak{g}(\mathcal{P}))$, rows corresponding to basis elements of the form $\mathbf{E_{i,j}}$, for $i,j\in Ext(\mathcal{P})$, contribute $$|Rel_E(\mathcal{P})|-|Ext(\mathcal{P})|+1$$ to the index of $\mathfrak{g}_A(\mathcal{P})$. Furthermore, Theorem~\ref{thm:commat} can be restated for type-A Lie poset algebras $\mathfrak{g}_A(\mathcal{P})$ as
$$\ind(\mathfrak{g}_A(\mathcal{P}))=\dim(B_3(\mathcal{P}))-Rank_{R(\mathfrak{g})}(B_3(\mathcal{P}))+|Rel_E(\mathcal{P})|-|Ext(\mathcal{P})|+1,$$ where $B_3(\mathcal{P})$ denotes the $B_3$ block of $C'(\mathfrak{g}(\mathcal{P}))$.
\end{remark}

\begin{remark}\label{rem:h2blocks}
When $\mathcal{P}$ is of height two, the $B_3$  block of $C'(\mathfrak{g}(\mathcal{P}))$ is block diagonal with each block on the diagonal corresponding to $i\in\mathcal{P}\backslash Ext(\mathcal{P})$. These blocks are defined by rows $\mathbf{E_{k_1,i}}$, $\mathbf{E_{i,i}}$, and $\mathbf{E_{i,k_2}}$, for $k_1,k_2\in \mathcal{P}$ such that $k_1$ is non-minimal in $\mathbb{Z}$ satisfying $k_1\preceq_{\mathcal{P}}i$.  To see that the blocks formed by these rows are disjoint, note that such rows have nonzero entries in columns of the form $\mathbf{E_{j_1,i}}$ and $\mathbf{E_{i,j_2}}$, for $j_1,j_2\in\mathcal{P}$. Such a block is highlighted in red in Figure~\ref{fig:commatrixrr}. We denote the block corresponding to $i\in\mathcal{P}\backslash Ext(\mathcal{P})$ in $B_3$ by $B_3(\mathcal{P},i)$.
\end{remark}

\subsection{Index formulas}\label{sec:if}

\begin{theorem}\label{h0ind}
If $\mathcal{P}$ is a height-zero poset, then 

\begin{eqnarray}\label{one}
\ind(\mathfrak{g}_A(\mathcal{P}))=|\mathcal{P}|-1.
\end{eqnarray}
\end{theorem}

\begin{proof}
Since $\mathfrak{g}_A(\mathcal{P})$ is necessarily commutative, equation (\ref{one}) follows from Theorem~\ref{thm:commat}.
\end{proof}

\begin{theorem}\label{h1ind}
If $\mathcal{P}$ is a connected, height-one poset, then 

\begin{eqnarray}\label{height1}
\ind(\mathfrak{g}_A(\mathcal{P}))=|Rel(\mathcal{P})|-|\mathcal{P}|+1.
\end{eqnarray}

\end{theorem}
\begin{proof}
Since every element of $\mathcal{P}$ is an element of $Ext(\mathcal{P})$, the matrix $C'(\mathfrak{g}(\mathcal{P}))$ contains no $B_3$ block and equation (\ref{height1}) follows from Remark~\ref{rem:contminmax}.
\end{proof}

\begin{Ex}\label{ex:h1}
Consider the height-one posets illustrated in Figure~\ref{fig:non}.  
Note that
$|\mathcal{P}(1,n)|=|\mathcal{P}(n,1)|=n+1$ and $|Rel(\mathcal{P}(1,n))|=|Rel(\mathcal{P}(n,1))|=n$.  It now follows from Theorem \ref{h1ind} that $\mathfrak{g}_A(\mathcal{P}(1,n))$ and $\mathfrak{g}_A(\mathcal{P}(n,1))$ are Frobenius.
\end{Ex}

To describe index formulas in height two requires a bit more notation.  The following definition applies to a poset of any height.

\begin{definition}
Let $\mathcal{P}$ be a poset and $j\in \mathcal{P}$.  Define 

$$D(\mathcal{P},j)=|\{i\in \mathcal{P}|i\preceq j\}|,$$ 

$$U(\mathcal{P},j)=|\{i\in \mathcal{P}|j\preceq i\}|,$$ and 

$$UD(\mathcal{P},j) =  \begin{cases} 
      |U(\mathcal{P},j)-D(\mathcal{P},j)|, & U(\mathcal{P},j)\neq D(\mathcal{P},j); \\
      2, & \text{otherwise.}
   \end{cases}
$$
\end{definition}

\begin{Ex}
If $\mathcal{P} = \mathcal{P}(1,1,2)$, then $U(\mathcal{P},2)=2$, $D(\mathcal{P},2)=1$, and $UD(\mathcal{P},2)=1$.
\end{Ex}

We now have the following concise result.

\begin{theorem}\label{thm:h2pure}
If $\mathcal{P}$ is a connected, height-two poset, then 

\begin{eqnarray}\label{height2}
\ind(\mathfrak{g}_A(\mathcal{P}))=Rel_E(\mathcal{P})-|\mathcal{P}|+1+\sum_{j\in \mathcal{P}\backslash Ext(\mathcal{P})}UD(\mathcal{P},j).
\end{eqnarray}
\end{theorem}
\begin{proof}
By Remark~\ref{rem:contminmax}, the rows $\mathbf{E_{i,j}}$ for $i,j\in Ext(\mathcal{P})$ in $C'(\mathfrak{g}(\mathcal{P}))$ contribute $$|Rel_E(\mathcal{P})|-|Ext(\mathcal{P})|+1$$ to the index of $\mathfrak{g}_A(\mathcal{P})$. Thus, in order to compute $\ind(\mathfrak{g}_A(\mathcal{P}))$ we need to determine the rank of the $B_3$ block of $C'(\mathfrak{g}(\mathcal{P}))$. However, by Remark~\ref{rem:h2blocks}, this corresponds to determining the ranks of the sub-matrices $B_3(\mathcal{P},i)$, for each $i\in\mathcal{P}\backslash Ext(\mathcal{P})$. Define $\mathcal{P}^i$, for $i\in\mathcal{P}\backslash Ext(\mathcal{P})$, to be the induced subposet of $\mathcal{P}$ generated by 

$$\{j\in\mathcal{P}~|~i\preceq j\text{ or }j\preceq i\text{ or }j=i\}.$$ 

\noindent
Note that $\mathcal{P}^i$ must be of the form $\mathcal{P}(n_i,1,m_i)$, where $n_i=D(\mathcal{P},i)$ and $m_i=U(\mathcal{P},i)$. Furthermore, $B_3(\mathcal{P},i)=B_3(\mathcal{P}^i,i)$. So, we need only determine the contribution $B_3(\mathcal{P}^i,i)$ makes to the rank of $C'(\mathfrak{g}_A(\mathcal{P}(n_i,1,m_i)))$, and thus to the index of $\mathfrak{g}_A(\mathcal{P}(n_i,1,m_i))$. 

We make use of following important result, whose proof is relegated to Appendix B.

\begin{theorem}\label{thm:CPn1m}

$$\ind\mathfrak{g}(\mathcal{P}(n,1,m)) =  \begin{cases} 
      n^2-2n+2, & n=m; \\
      n(m-2), & m>n; \\
      m(n-2), & n>m.
   \end{cases}
$$
\end{theorem}

\noindent
Assuming Theorem~\ref{thm:CPn1m} and removing $$|Rel_E(\mathcal{P}(n_i,1,m_i))|-|Ext(\mathcal{P}(n_i,1,m_i))|+1$$ from $\ind(\mathfrak{g}_A(\mathcal{P}(n_i,1,m_i))$, we find that the contribution $B_3(\mathcal{P},i)$ makes to the index of $\mathfrak{g}_A(\mathcal{P})$ is given by the following formula:

$$ 
B_3(\mathcal{P},i) \textit{ index contribution}=
\begin{cases} 
      1, & n_i=m_i; \\
      m_i-n_i-1, & m_i>n_i; \\
      n_i-m_i-1, & n_i>m_i;
\end{cases}$$ 
that is, $B_3(\mathcal{P},i)$ contributes $UD(\mathcal{P},i)-1$ to the index of $\mathfrak{g}_A(\mathcal{P})$. Thus, $$\ind(\mathfrak{g}_A(\mathcal{P}))=Rel_E(\mathcal{P})-|Ext(\mathcal{P})|+1+\sum_{j\in\mathcal{P}\backslash Ext(\mathcal{P})}(UD(\mathcal{P},j)-1).$$ Now, for each element $i\in \mathcal{P}\backslash Ext(\mathcal{P})$, adding $-1$ to $-|Ext(\mathcal{P})|$ and $+1$ to the contribution from $B_3(\mathcal{P},i)$  establishes~(\ref{height2}).
\end{proof}

\begin{Ex}\label{ex:h2}
The height-two poset $\mathcal{P}$ in Example~\ref{Stargate} 
has $$Rel_E(\mathcal{P})=2,\quad |\mathcal{P}|=4,\quad \text{and}\quad UD(\mathcal{P},3)=1.$$ By Theorem~\ref{thm:h2pure}, the algebra $\mathfrak{g}_A(\mathcal{P})$ is Frobenius.
\end{Ex}

\noindent
The following theorem is an immediate corollary of Theorems~\ref{h1ind} and~\ref{thm:h2pure}, and will be needed in the characterization of height-two posets corresponding to Frobenius Lie poset algebras. See Theorem~\ref{thm:npurechar}.

\begin{theorem}\label{cor:minmax}
Let $\mathcal{P}$ and $\mathcal{Q}$ be posets of height two or less. If $\mathcal{R}$ denotes the poset obtained by identifying a minimal element of $\mathcal{P}$ with a minimal element of $\mathcal{Q}$, or identifying a maximal element of $\mathcal{P}$ and $\mathcal{Q}$, then 

$$\ind(\mathfrak{g}_A(\mathcal{R}))= \ind(\mathfrak{g}_A(\mathcal{P}))+\ind(\mathfrak{g}_A(\mathcal{Q})).
$$

\end{theorem}

\begin{remark}\label{rem:minmax}
The result of Corollary~\ref{cor:minmax} holds more generally. In particular, if $\mathcal{P}$ and $\mathcal{Q}$ are posets of arbitrary height, and $\mathcal{R}$ denotes the poset obtained by identifying a minimal element of $\mathcal{P}$ with a minimal element of $\mathcal{Q}$, or identifying a maximal element of $\mathcal{P}$ and $\mathcal{Q}$, then 

$$\ind(\mathfrak{g}_A(\mathcal{R}))= \ind(\mathfrak{g}_A(\mathcal{P}))+\ind(\mathfrak{g}_A(\mathcal{Q})).
$$

\end{remark}

To remove the connected restriction used throughout this section, we make use of the following theorem.

\begin{theorem}\label{thm:disjoint}
If $\{\mathcal{P}_i\}_{i=1}^n$ is a collection of connected posets on pairwise disjoint sets, then $$\ind\bigg(\mathfrak{g}_A\bigg(\sum_{i=1}^n\mathcal{P}_i\bigg)\bigg)=\bigg[\sum_{i=1}^n \ind(\mathfrak{g}_A(\mathcal{P}_i))\bigg]+n-1.$$
\end{theorem}
\begin{proof}
For this proof we will use the basis $\mathscr{B}$ of $\mathfrak{g}(\mathcal{P})$ with $\sum\mathbf{E_{i,i}}$ replaced by $\mathbf{E_{1,1}}$. First, label the rows and columns of $C(\mathfrak{g}(\sum_{i=1}^n\mathcal{P}_i))$ by the basis elements corresponding to $\mathcal{P}_1$, followed by those of $\mathcal{P}_2$, etc. In this way, $C(\mathfrak{g}(\sum_{i=1}^n\mathcal{P}_i))$ becomes block diagonal with the $i$th block corresponding to $C(\mathfrak{g}(\mathcal{P}_i))$.
\end{proof}

The following corollary to Theorem~\ref{thm:disjoint} is immediate.

\begin{theorem} If $\mathfrak{g}_A(\mathcal{P})$ is Frobenius, then $\mathcal{P}$ is connected.
\end{theorem}

\begin{Ex}
Applying Theorem~\ref{thm:disjoint} to the posets in Example~\ref{ex:h1} and Example~\ref{ex:h2} gives that 

$$\ind(\mathfrak{g}_A(\mathcal{P}(1,2)+\mathcal{P}(1,1,2)))=1.$$
\end{Ex}

\noindent
Putting the results of this section together yields the following beautiful formula.

\begin{theorem}\label{thm:gind}
If $\mathcal{P}$ is a poset of height one or two, then 

\begin{eqnarray}\label{eqn:grind}
\ind(\mathfrak{g}_A(\mathcal{P}))=Rel_E(\mathcal{P})-|\mathcal{P}|+2\cdot C_{\mathcal{P}}-1+\sum_{j\in \mathcal{P}\backslash Ext(\mathcal{P})}UD(\mathcal{P},j).
\end{eqnarray} 
\end{theorem}

\section{Combinatorial classification of Frobenius posets of restricted height - gluing rules}\label{sec:frobchar}

\textbf{N.B.}  To streamline the narrative, in this section and the next, we will often refer to posets corresponding to type-A Frobenius Lie poset algebras as \textit{Frobenius posets}. We will remove this convention after considering the other classical types in Section 6.  

\bigskip
In this section, we characterize Frobenius posets of height zero, one, and two.  The characterization of Frobenius posets of heights 
zero and one are straightforward and follow from Theorems~\ref{h0ind},~\ref{h1ind}, and~\ref{thm:disjoint}. 

\begin{theorem}\label{thm:h01FC}
If $\mathcal{P}$ is a poset of height zero or one,  then $\mathfrak{g}_A(\mathcal{P})$ is Frobenius if and only if the Hasse diagram of $\mathcal{P}$ is a tree.
\end{theorem}

The characterization of Frobenius posets of height two is nontrivial. To start, we first characterize Frobenius, pure, height-two posets $\mathcal{P}$.

\begin{remark}\label{rem:constph2}
As mentioned in the proof of Theorem~\ref{thm:h2pure}, each rank-one element $i\in\mathcal{P}$ defines a poset $\mathcal{P}^i$ of the form $\mathcal{P}(n_i,1,m_i)$, where $n_i=D(\mathcal{P},i)$ and $m_i=U(\mathcal{P},i)$. Thus, any pure, height-two poset $\mathcal{P}$ can be constructed via a ``gluing" process, starting from the collection of posets $\{\mathcal{P}^i\}_{i\in\mathcal{P}}$ indexed by the rank-one elements of $\mathcal{P}$, and sequentially identifying minimal and maximal elements of each.  This gluing process is illustrated in Example~\ref{ex:gluing}. Moreover, if the poset is connected, this process can be performed in such a way that the resulting poset is connected at each step.
\end{remark}

\begin{lemma}\label{lem:onlystar}
If $\mathcal{P}$ is a Frobenius, pure, height-two poset, then for each rank-one element $i\in\mathcal{P}$, $\mathcal{P}^i$ is of the form $\mathcal{P}(2,1,1)$ or $\mathcal{P}(1,1,2)$.
\end{lemma}
\begin{proof}
It follows from Theorem~\ref{thm:CPn1m} that
\[\ind\mathfrak{g}(\mathcal{P}(n,1,m)) =  \begin{cases} 
      n^2-2n+2, & n=m; \\
      n(m-2), & m>n; \\
      m(n-2), & n>m.
   \end{cases}
\]
Thus, if $\mathcal{P}$ is a Frobenius, pure, height-two poset with a single rank-one element, then $\mathcal{P}$ is of the form $\mathcal{P}(2,1,1)$ or $\mathcal{P}(1,1,2)$.

Consider a pure, height-two poset $\mathcal{Q}$ with $r_1>1$ rank-one elements, which we label $i_1,\hdots,i_{r_1}$.  We want to show that if $\mathcal{Q}$ is Frobenius, then $\mathcal{P}^{i_j}$ is of the form $\mathcal{P}(1,1,2)$ or $\mathcal{P}(2,1,1)$ for $1\le j\le r_1$.  Assume for a contradiction that there exists $i_k$, for $1\le k\le r_1$, satisfying $\mathcal{P}^{i_k}$ is of the form $\mathcal{P}(n_k,1,m_k)$, where $(n_k,m_k)\neq (1,2),(2,1)$ as ordered pairs. By Remark~\ref{rem:constph2}, one can construct $\mathcal{Q}$ starting from $\mathcal{P}^{i_k}$ and inductively adjoining the posets -- by identifying minimal or maximal elements -- $\mathcal{P}^{i_j}$, which are of the form $\mathcal{P}(n_j,1,m_j)$, for $1\le j\neq k\le n$ . Furthermore, this construction can be performed in such a way that at each stage the resulting poset is connected. Thus, it suffices to show that adjoining $\mathcal{S}$ of the form $\mathcal{P}(n,1,m)$ to a pure, connected, height-two poset $\mathcal{P}'$ by identifying minimal or maximal elements cannot result in a poset with smaller index. Denote by $\mathcal{P}$ the poset obtained by combining $\mathcal{P}'$ and $\mathcal{S}$. Without loss of generality, assume that we have identified at least one maximal element of $\mathcal{P}'$ with a maximal element of $\mathcal{S}$ and let $p_1$ be the common label. Let $p_2$ be an arbitrary minimal element of $\mathcal{S}$. Assume $\mathcal{S}$ contains $k$ maximal and $l$ minimal elements which are not identified with elements of $\mathcal{P}'$ in $\mathcal{P}$. Using Theorem~\ref{thm:h2pure}, $$\ind(\mathfrak{g}_A(\mathcal{P}))=\ind(\mathfrak{g}_A(\mathcal{P}'))+|Rel_E(\mathcal{P})\backslash Rel_E(\mathcal{P}')|-(k+l+1)+1+UD(\mathcal{S},s)$$ $$=\ind(\mathfrak{g}_A(\mathcal{P}'))+|Rel_E(\mathcal{P})\backslash Rel_E(\mathcal{P}')|-k-l+UD(\mathcal{S},s),$$ where $s$ is the rank-one element of $\mathcal{S}$. As $UD(\mathcal{S},s)\ge 0$, we need only show that $$|Rel_E(\mathcal{P})\backslash Rel_E(\mathcal{P}')|\ge k+l.$$ To this end, $p_1$ is related to all $l$ new minimal elements and $p_2$ is related to all $k$ new maximal elements. Note, there is no overlap in these two collections of relations in $Ext(\mathcal{P})$ since we assumed $p_1\in \mathcal{P}'$. Therefore, $|Rel_E(\mathcal{P})\backslash Rel_E(\mathcal{P}')|\ge k+l$ which implies that $\ind(\mathfrak{g}_A(\mathcal{P}))\ge \ind(\mathfrak{g}_A(\mathcal{P}'))$ and the result follows.
\end{proof}

We now determine how to identify minimal and maximal elements of posets of the form $\mathcal{P}(2,1,1)$ and $\mathcal{P}(1,1,2)$ so that the resulting poset is Frobenius. Let $\mathcal{Q}$ be a pure, height-two poset and $\mathcal{S}$ a poset of the form $\mathcal{P}(1,1,2)$ or $\mathcal{P}(2,1,1)$. Let $\mathcal{S}$ have elements $a_1,a_2,b,$ and $c$ with either $a_1,a_2\preceq b\preceq c$ or $c\preceq b\preceq a_1,a_2$. Furthermore, let $x,y,z\in Ext(\mathcal{Q})$. To fix notation, assume that if $c, a_1$, or $a_2$ is identified with an element of $Ext(\mathcal{Q})$, then that element must be $x,y$, or $z$, respectively. The following Table~\ref{tab:h2fassem} lists all possible ways (``gluing rules") of identifying the elements $a,b,c\in\mathcal{S}$ with the elements $x,y,z\in\mathcal{Q}$. The last column of Table~\ref{tab:h2fassem} records the attendant contributions to the index; that is, if $\mathcal{P}$ is the poset resulting from gluing $\mathcal{S}$ to $\mathcal{Q}$, then this column gives $\ind(\mathfrak{g}_A(\mathcal{P}))-\ind(\mathfrak{g}_A(\mathcal{Q}))$.

\begin{lemma}\label{lem:table}
The table below summarizes the contribution to the index of a pure height-two poset upon gluing a copy of $\mathcal{P}(1,1,2)$ or $\mathcal{P}(2,1,1)$ as described above. Convention: Let $\sim$ denote that two elements are related;  that is, $x\sim y$ means $x\preceq y$ or $y\preceq x$.

\begin{table}[H]
\centering
\begin{tabular}{c|c|c|c|c}
Gluing Rule & $c$                       & $a_1$                         & $a_2$                         & Contribution to the Index \\ \hline
$A_1$ & $c\neq x$ & $a_1=y$         & $a_2\neq z$ & 0                         \\
$A_2$ & $c\neq x$ & $a_1\neq y$             & $a_2=z$ & 0                         \\
$B$ & $c\neq x$ & $a_1=y$       & $a_2=z$                 & 1                         \\
$C$ & $c=x$ & $a_1\neq y$                 & $a_2\neq z$                 & 0                         \\
$D_1$ & $c=x$                 & $a_1=y$, $y\sim x$ & $a_2\neq z$  & 0                         \\
$D_2$ & $c=x$                 & $a_1\neq y$  & $a_2= z$, $z\sim x$  & 0                         \\
$E_1$ & $c=x$                 & $a_1= y$, $y\nsim x$  & $a_2\neq z$  & 1                         \\
$E_2$ & $c=x$                 & $a_1\neq y$  & $a_2= z$, $z\nsim x$  & 1                         \\
$F$ & $c=x$                 & $a_1= y$, $y\sim x$  & $a_2= z$, $z\sim x$  & 0                         \\
$G_1$ & $c=x$                 & $a_1= y$, $y\sim x$  & $a_2= z$, $z\nsim x$  & 1                         \\
$G_2$ & $c=x$                 & $a_1= y$, $y\nsim x$  & $a_2= z$, $z\sim x$  & 1                         \\
$H$ & $c=x$                 & $a_1= y$, $y\nsim x$  & $a_2= z$, $z\nsim x$  & 2                         \\
\end{tabular}
\caption{Height-two gluing rules}\label{tab:h2fassem}
\end{table}

\end{lemma}
\begin{proof}
Follows directly using Theorem~\ref{thm:h2pure}.
\end{proof}

\begin{Ex}\label{ex:gluing}
In Figure~\ref{fig:indinv}, we illustrate the Table 1 gluing rules which do not alter the index.
\begin{figure}[H]
$$\begin{tikzpicture}[scale=0.6]
\node [circle, draw = black, fill = black, inner sep = 0.5mm] (v1) at (-2.5,0.5) {};
\node [circle, draw = black, fill = black, inner sep = 0.5mm] (v2) at (-2.5,1.5) {};
\node [circle, draw = black, fill = black, inner sep = 0.5mm] (v4) at (-3,2.5) {};
\node [circle, draw = black, fill = black, inner sep = 0.5mm] (v3) at (-2,2.5) {};
\draw (v1) -- (v2) -- (v3);
\draw (v2) -- (v4);
\draw[->] (-1.5,1.5) -- (-0.5,1.5);
\node [circle, draw = black, fill = black, inner sep = 0.5mm] (v7) at (0,2.5) {};
\node [circle, draw = black, fill = black, inner sep = 0.5mm] (v8) at (1,2.5) {};
\node [circle, draw = black, fill = black, inner sep = 0.5mm] (v11) at (2,2.5) {};
\node [circle, draw = black, fill = black, inner sep = 0.5mm] (v10) at (1.5,1.5) {};
\node [circle, draw = black, fill = black, inner sep = 0.5mm] (v6) at (0.5,1.5) {};
\node [circle, draw = black, fill = black, inner sep = 0.5mm] (v5) at (0.5,0.5) {};
\node [circle, draw = black, fill = black, inner sep = 0.5mm] (v9) at (1.5,0.5) {};
\node at (-1,2) {$A_1$};
\draw (v5) -- (v6) -- (v7);
\draw (v6) -- (v8);
\draw (v9) -- (v10) -- (v8);
\draw (v10) -- (v11);
\draw [->] (2.5,1.5) -- (3.5,1.5);
\node at (3,2) {$C$};
\node [circle, draw = black, fill = black, inner sep = 0.5mm] (v14) at (4,2.5) {};
\node [circle, draw = black, fill = black, inner sep = 0.5mm] (v15) at (5,2.5) {};
\node [circle, draw = black, fill = black, inner sep = 0.5mm] (v18) at (6,2.5) {};
\node [circle, draw = black, fill = black, inner sep = 0.5mm] (v20) at (7,2.5) {};
\node [circle, draw = black, fill = black, inner sep = 0.5mm] (v21) at (8,2.5) {};
\node [circle, draw = black, fill = black, inner sep = 0.5mm] (v13) at (4.5,1.5) {};
\node [circle, draw = black, fill = black, inner sep = 0.5mm] (v17) at (5.5,1.5) {};
\node [circle, draw = black, fill = black, inner sep = 0.5mm] (v19) at (6.5,1.5) {};
\node [circle, draw = black, fill = black, inner sep = 0.5mm] (v12) at (4.5,0.5) {};
\node [circle, draw = black, fill = black, inner sep = 0.5mm] (v16) at (5.5,0.5) {};
\draw (v12) -- (v13) -- (v14);
\draw (v13) -- (v15);
\draw (v16) -- (v17) -- (v15);
\draw (v17) -- (v18);
\draw (v16) -- (v19) -- (v20);
\draw (v19) -- (v21);
\draw[->] (8.5,1.5) -- (9.5,1.5);
\node at (9,2) {$D_1$};
\node  (v24) at (10,2.5) [circle, draw = black, fill = black, inner sep = 0.5mm] {};
\node  (v25) at (11,2.5) [circle, draw = black, fill = black, inner sep = 0.5mm] {};
\node (v28) at (12,2.5) [circle, draw = black, fill = black, inner sep = 0.5mm] {};
\node (v31) at (13,2.5) [circle, draw = black, fill = black, inner sep = 0.5mm] {};
\node (v32) at (14,2.5) [circle, draw = black, fill = black, inner sep = 0.5mm] {};
\node  (v23) at (10.5,1.5) [circle, draw = black, fill = black, inner sep = 0.5mm] {};
\node  (v27) at (11.5,1.5) [circle, draw = black, fill = black, inner sep = 0.5mm] {};
\node  (v29) at (13.5,1.5) [circle, draw = black, fill = black, inner sep = 0.5mm] {};
\node  (v22) at (10.5,0.5) [circle, draw = black, fill = black, inner sep = 0.5mm] {};
\node  (v26) at (11.5,0.5) [circle, draw = black, fill = black, inner sep = 0.5mm] {};
\node  (v30) at (12.5,1.5) [circle, draw = black, fill = black, inner sep = 0.5mm] {};
\draw (v22) -- (v23) -- (v24);
\draw (v23) -- (v25);
\draw (v26) -- (v27) -- (v25);
\draw (v27) -- (v28);
\draw (v26) -- (v29) -- (v28);
\draw (v26) -- (v30);
\draw (v31) -- (v30) -- (v32);
\node [circle, draw = black, fill = black, inner sep = 0.5mm] (v33) at (15,2.5) {};
\draw (v29) -- (v33);
\draw [->] (15.5,1.5) -- (16.5,1.5);
\node at (16,2) {$F$};
\node [circle, draw = black, fill = black, inner sep = 0.5mm] (v46) at (17,1.5) {};
\node [circle, draw = black, fill = black, inner sep = 0.5mm] (v36) at (17.5,2.5) {};
\node [circle, draw = black, fill = black, inner sep = 0.5mm] (v37) at (18.5,2.5) {};
\node [circle, draw = black, fill = black, inner sep = 0.5mm] (v40) at (19.5,2.5) {};
\node [circle, draw = black, fill = black, inner sep = 0.5mm] (v42) at (20.5,2.5) {};
\node [circle, draw = black, fill = black, inner sep = 0.5mm] (v43) at (21.5,2.5) {};
\node [circle, draw = black, fill = black, inner sep = 0.5mm] (v45) at (22.5,2.5) {};
\node [circle, draw = black, fill = black, inner sep = 0.5mm] (v35) at (18,1.5) {};
\node [circle, draw = black, fill = black, inner sep = 0.5mm] (v39) at (19,1.5) {};
\node [circle, draw = black, fill = black, inner sep = 0.5mm] (v41) at (20,1.5) {};
\node [circle, draw = black, fill = black, inner sep = 0.5mm] (v44) at (21,1.5) {};
\node [circle, draw = black, fill = black, inner sep = 0.5mm] (v34) at (18,0.5) {};
\node [circle, draw = black, fill = black, inner sep = 0.5mm] (v38) at (19,0.5) {};
\draw (v34) -- (v35) -- (v36);
\draw (v35) -- (v37);
\draw (v38) -- (v39) -- (v37);
\draw (v39) -- (v40);
\draw (v38) -- (v41) -- (v42);
\draw (v41) -- (v43);
\draw (v38) -- (v44) -- (v40);
\draw (v44) -- (v45);
\draw (v34) -- (v46) -- (v36);
\draw (v46) -- (v37);
\node at (1.5,0) {$c$};
\node at (1,3) {$a_1$};
\node at (2,3) {$a_2$};
\node at (5.5,0) {$c$};
\node at (7,3) {$a_1$};
\node at (8,3) {$a_2$};
\node at (11.5,0) {$c$};
\node at (12,3) {$a_1$};
\node at (15,3) {$a_2$};
\node at (18,0) {$c$};
\node at (17.5,3) {$a_1$};
\node at (18.5,3) {$a_2$};
\end{tikzpicture}$$
\caption{Index non-altering gluing rules}\label{fig:indinv}
\end{figure}
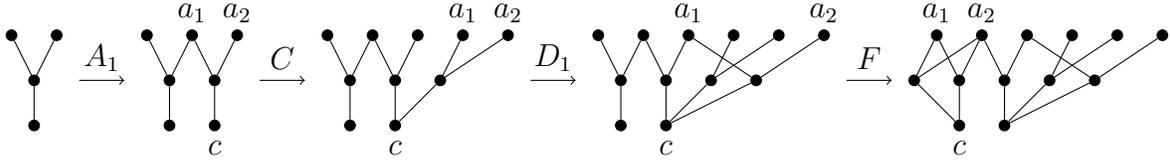
\end{Ex}

\begin{Ex}
In Figure~\ref{fig:indalt}, we illustrate the Table 1 gluing rules which do alter the index. These correspond to gluing rules $B$, $E_1$ \textup($E_2$ is similar\textup), $G_1$ \textup($G_2$ is similar\textup), and $H$, respectively. 
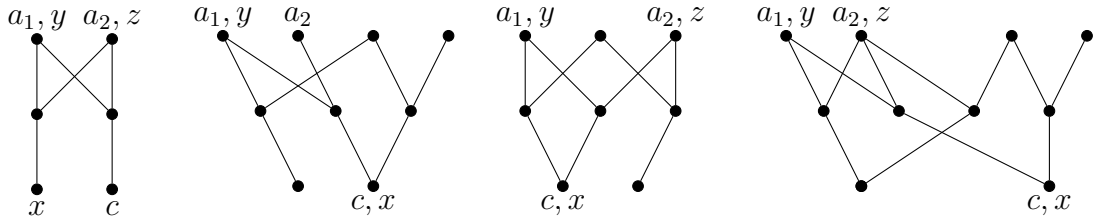
\begin{figure}[H]
$$\begin{tikzpicture}
\node (v6) at (-2.5,2) [circle, draw = black, fill = black, inner sep = 0.5mm] {};
\node (v3) at (-1.5,2) [circle, draw = black, fill = black, inner sep = 0.5mm] {};
\node (v2) at (-2.5,1) [circle, draw = black, fill = black, inner sep = 0.5mm] {};
\node (v5) at (-1.5,1) [circle, draw = black, fill = black, inner sep = 0.5mm] {};
\node (v1) at (-2.5,0) [circle, draw = black, fill = black, inner sep = 0.5mm] {};
\node (v4) at (-1.5,0) [circle, draw = black, fill = black, inner sep = 0.5mm] {};
\draw (v1) -- (v2) -- (v3);
\draw (v4) -- (v5) -- (v6);
\draw (v6) -- (v2);
\draw (v5) -- (v3);
\node at (-2.5,2.25) {$a_1,y$};
\node at (-1.5,2.25) {$a_2,z$};
\node at (-2.5,-0.25) {$x$};
\node at (-1.5,-0.25) {$c$};
\end{tikzpicture}\quad\begin{tikzpicture}
\node (v3) at (-2.5,2) [circle, draw = black, fill = black, inner sep = 0.5mm] {};
\node (v8) at (-1.5,2) [circle, draw = black, fill = black, inner sep = 0.5mm] {};
\node (v4) at (-0.5,2) [circle, draw = black, fill = black, inner sep = 0.5mm] {};
\node (v9) at (0.5,2) [circle, draw = black, fill = black, inner sep = 0.5mm] {};
\node (v7) at (0,1) [circle, draw = black, fill = black, inner sep = 0.5mm] {};
\node (v6) at (-1,1) [circle, draw = black, fill = black, inner sep = 0.5mm] {};
\node (v2) at (-2,1) [circle, draw = black, fill = black, inner sep = 0.5mm] {};
\node (v1) at (-1.5,0) [circle, draw = black, fill = black, inner sep = 0.5mm] {};
\node (v5) at (-0.5,0) [circle, draw = black, fill = black, inner sep = 0.5mm] {};
\draw (v1) -- (v2) -- (v3);
\draw (v2) -- (v4);
\draw (v5) -- (v6) -- (v3);
\draw (v5) -- (v7) -- (v4);
\draw (v6) -- (v8);
\draw (v7) -- (v9);
\node at (-2.5,2.25) {$a_1,y$};
\node at (-0.5,-0.25) {$c,x$};
\node at (-1.5,2.25) {$a_2$};
\end{tikzpicture}\quad\begin{tikzpicture}
\node (v3) at (-0.5,2) [circle, draw = black, fill = black, inner sep = 0.5mm] {};
\node (v5) at (0.5,2) [circle, draw = black, fill = black, inner sep = 0.5mm] {};
\node (v7) at (1.5,2) [circle, draw = black, fill = black, inner sep = 0.5mm] {};
\node (v8) at (1.5,1) [circle, draw = black, fill = black, inner sep = 0.5mm] {};
\node (v4) at (0.5,1) [circle, draw = black, fill = black, inner sep = 0.5mm] {};
\node (v2) at (-0.5,1) [circle, draw = black, fill = black, inner sep = 0.5mm] {};
\node (v1) at (0,0) [circle, draw = black, fill = black, inner sep = 0.5mm] {};
\node (v6) at (1,0) [circle, draw = black, fill = black, inner sep = 0.5mm] {};
\draw (v1) -- (v2) -- (v3);
\draw (v1) -- (v4) -- (v3);
\draw (v6) --  (v8) -- (v7);
\draw (v8) -- (v5);
\draw (v2) -- (v5);
\draw (0.5,1) -- (v7);
\node at (0,-0.25) {$c,x$};
\node at (-0.5,2.25) {$a_1,y$};
\node at (1.5,2.25) {$a_2,z$};
\end{tikzpicture}\quad\begin{tikzpicture}
\node (v10) at (-3.5,2) [circle, draw = black, fill = black, inner sep = 0.5mm]  {};
\node (v7) at (-2.5,2) [circle, draw = black, fill = black, inner sep = 0.5mm]  {};
\node (v4) at (-0.5,2) [circle, draw = black, fill = black, inner sep = 0.5mm]  {};
\node (v3) at (0.5,2) [circle, draw = black, fill = black, inner sep = 0.5mm]  {};
\node (v2) at (0,1) [circle, draw = black, fill = black, inner sep = 0.5mm]  {};
\node (v6) at (-1,1) [circle, draw = black, fill = black, inner sep = 0.5mm]  {};
\node (v9) at (-3,1) [circle, draw = black, fill = black, inner sep = 0.5mm]  {};
\node [circle, draw = black, fill = black, inner sep = 0.5mm] (v8) at (-2.5,0) {};
\node [circle, draw = black, fill = black, inner sep = 0.5mm] (v5) at (-2.5,0) {};
\node (v1) at (0,0) [circle, draw = black, fill = black, inner sep = 0.5mm]  {};
\draw (v1) -- (v2) -- (v3);
\draw (v2) -- (v4);
\draw (v5) -- (v6) -- (v4);
\draw (v6) -- (v7);
\draw (v8) -- (v9) -- (v7);
\draw (v9) -- (v10);
\node (v11) at (-2,1) [circle, draw = black, fill = black, inner sep = 0.5mm]  {};
\draw (v1) -- (v11) -- (v7);
\draw (v11) -- (v10);
\node at (0,-0.25) {$c,x$};
\node at (-3.5,2.25) {$a_1,y$};
\node at (-2.5,2.25) {$a_2,z$};
\end{tikzpicture}$$
\caption{Index altering gluing rules}\label{fig:indalt}
\end{figure}
\end{Ex}

The following theorem is a result of Lemma~\ref{lem:onlystar} and Lemma~\ref{lem:table}.

\begin{theorem}\label{thm:h2frobchar}
Any Frobenius, pure, height-two poset is contructed from copies of $\mathcal{P}(2,1,1)$ or $\mathcal{P}(1,1,2)$ by applying gluing rules $A_1, A_2, C, D_1, D_2$, or $F$ of Table~\ref{tab:h2fassem}.
\end{theorem}

It remains to characterize Frobenius, non-pure, height-two posets. Note that a height-two poset $\mathcal{P}$ is non-pure if and only if there exists covering relations between minimal and maximal elements of $\mathcal{P}$. Removing such covering relations, leaves a disjoint union of 
singleton posets and Frobenius, pure, height-two posets. Let $\mathscr{P}=\{\mathcal{P}_1,\hdots,\mathcal{P}_n\}$ be the resulting collection of pure, height-two posets.  Elements of $\mathscr{P}$ 
will be called \textit{pure components} of $\mathcal{P}$.  Since $\mathcal{P}$ is connected, each pure component of $\mathcal{P}$ is connected to every other pure component in the Hasse diagram of $\mathcal{P}$.  In particular, the pure components of $\mathcal{P}$ are connected by paths which alternate between paths consisting of covering relations between elements of $Ext(\mathcal{P})$ and paths contained in pure components of $\mathcal{P}$. To characterize Frobenius, non-pure, height-two posets, we will outline an inductive procedure for constructing $\mathcal{P}$ from its pure components and covering relations between elements of $Ext(\mathcal{P})$. Such a construction breaks into two stages. First, we will construct a subposet $\mathcal{P}'$ starting from any pure component, say $\mathcal{P}_1$ of $\mathcal{P}$, which is, in a sense, a minimally connected subposet of $\mathcal{P}$ containing all elements and pure components of $\mathcal{P}$. Let $\mathcal{Q}_i$ denote the subposet of $\mathcal{P}'$ formed at stage $i$ so that $\mathcal{Q}_1=\mathcal{P}_1$. Given $\mathcal{Q}_{i-1}$, the poset $\mathcal{Q}_i$ is formed as follows:
\begin{enumerate}
    \item add all covering relations of $\mathcal{P}$ between pairs of elements consisting of a minimal (or maximal) element of $\mathcal{Q}_{i-1}$ and a unique maximal (or minimal) element of $\mathcal{P}\backslash\mathcal{Q}_{i-1}$;
    \item next, form $\mathcal{Q}_i$ by adding all pure components of $\mathcal{P}$ sharing a single minimal or maximal element with the poset formed in step 1 above. Here, the ``addition" is accomplished by identifying the given shared elements of $\mathcal{P}$.
\end{enumerate}
Since $\mathcal{P}$ is finite, there must exist $m$ for which $\mathcal{Q}_m$ cannot be extended to another subposet of $\mathcal{P}$ by applying rules 1 or 2 above. Set $\mathcal{P}'=\mathcal{Q}_m$.

\begin{Ex}\label{ex:npurefrob}
A non-pure, height-two poset $\mathcal{P}$ along with the construction of a choice of $\mathcal{P}'$, as outlined above, is illustrated in Figure~\ref{fig:Frobnonpure}.
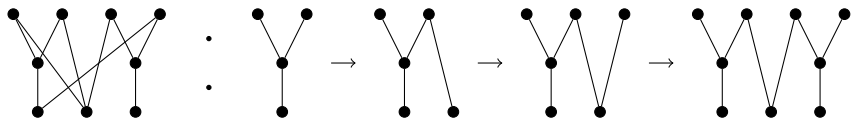
\begin{figure}[H]
$$\begin{tikzpicture}[scale=0.65]
\node [circle, draw = black, fill = black, inner sep = 0.5mm] (v1) at (-2,-0.5) {};
\node (v5) at (-1,-0.5) [circle, draw = black, fill = black, inner sep = 0.5mm] {};
\node (v3) at (-1,0.5) [circle, draw = black, fill = black, inner sep = 0.5mm] {};
\node (v2) at (-1.5,1.5) [circle, draw = black, fill = black, inner sep = 0.5mm] {};
\node [circle, draw = black, fill = black, inner sep = 0.5mm] (v4) at (-0.5,1.5) {};
\draw (v1) -- (v2) -- (v3) -- (v4);
\draw (v5) -- (v3);
\node (v6) at (-2.5,1.5) [circle, draw = black, fill = black, inner sep = 0.5mm] {};
\draw (v6) -- (v1);
\node (v7) at (-3,-0.5) [circle, draw = black, fill = black, inner sep = 0.5mm] {};
\node (v8) at (-3,0.5) [circle, draw = black, fill = black, inner sep = 0.5mm] {};
\node (v9) at (-3.5,1.5) [circle, draw = black, fill = black, inner sep = 0.5mm] {};
\draw (v7) -- (v8) -- (v9);
\draw (v8) -- (v6);
\draw (v1) -- (v9);
\draw (v7) -- (v4);

\node [circle, draw = black, fill = black, inner sep = 0.2mm] at (0.5,1) {};
\node [circle, draw = black, fill = black, inner sep = 0.2mm] at (0.5,0) {};
\node [circle, draw = black, fill = black, inner sep = 0.5mm] (v30) at (1.5,1.5) {};
\node [circle, draw = black, fill = black, inner sep = 0.5mm] (v31) at (2.5,1.5) {};
\node [circle, draw = black, fill = black, inner sep = 0.5mm] (v29) at (2,0.5) {};
\node [circle, draw = black, fill = black, inner sep = 0.5mm] (v28) at (2,-0.5) {};
\draw (v28) -- (v29) -- (v30);
\draw (v29) -- (v31);
\draw[->] (3,0.5) -- (3.5,0.5);
\node [circle, draw = black, fill = black, inner sep = 0.5mm] (v34) at (7,1.5) {};
\node [circle, draw = black, fill = black, inner sep = 0.5mm] (v35) at (8,1.5) {};
\node [circle, draw = black, fill = black, inner sep = 0.5mm] (v33) at (7.5,0.5) {};
\node [circle, draw = black, fill = black, inner sep = 0.5mm] (v32) at (7.5,-0.5) {};
\node [circle, draw = black, fill = black, inner sep = 0.5mm] (v36) at (8.5,-0.5) {};
\node [circle, draw = black, fill = black, inner sep = 0.5mm] (v37) at (9,1.5) {};
\draw (v32) -- (v33) -- (v34) -- cycle;
\draw (v33) -- (v35) -- (v36) -- (v37);
\draw[->] (9.5,0.5) -- (10,0.5);
\node [circle, draw = black, fill = black, inner sep = 0.5mm] (v45) at (10.5,1.5) {};
\node [circle, draw = black, fill = black, inner sep = 0.5mm] (v40) at (11.5,1.5) {};
\node [circle, draw = black, fill = black, inner sep = 0.5mm] (v39) at (11,0.5) {};
\node [circle, draw = black, fill = black, inner sep = 0.5mm] (v38) at (11,-0.5) {};
\node [circle, draw = black, fill = black, inner sep = 0.5mm] (v41) at (12,-0.5) {};
\node [circle, draw = black, fill = black, inner sep = 0.5mm] (v42) at (12.5,1.5) {};
\node [circle, draw = black, fill = black, inner sep = 0.5mm] (v46) at (13.5,1.5) {};
\node [circle, draw = black, fill = black, inner sep = 0.5mm] (v43) at (13,0.5) {};
\node [circle, draw = black, fill = black, inner sep = 0.5mm] (v44) at (13,-0.5) {};
\draw (v38) -- (v39) -- (v40) -- (v41) -- (v42) -- (v43) -- (v44);
\draw (v39) -- (v45);
\draw (v43) -- (v46);
\node (v14) at (4,1.5)  [circle, draw = black, fill = black, inner sep = 0.5mm] {};
\node (v12) at (5,1.5)  [circle, draw = black, fill = black, inner sep = 0.5mm] {};
\node (v11) at (4.5,0.5)  [circle, draw = black, fill = black, inner sep = 0.5mm] {};
\node (v10) at (4.5,-0.5)  [circle, draw = black, fill = black, inner sep = 0.5mm] {};
\node (v13) at (5.5,-0.5)  [circle, draw = black, fill = black, inner sep = 0.5mm] {};
\draw (v10) -- (v11) -- (v12) -- (v13);
\draw (v11) -- (v14);
\draw[->] (6,0.5) -- (6.5,0.5);
\end{tikzpicture}$$
\caption{Non-pure height two poset}\label{fig:Frobnonpure}
\end{figure}
\end{Ex}

\noindent
By construction, $\mathcal{P}'$ contains all elements of $\mathcal{P}$ and removing any covering relation between elements of $Ext(\mathcal{P}')=Ext(\mathcal{P})$  results in a disconnected poset. Furthermore, applying Corollary~\ref{cor:minmax} at each stage of the construction of $\mathcal{P}'$, the index of $\mathfrak{g}_A(\mathcal{P}')$ is equal to $\sum_{i=1}^n\ind(\mathfrak{g}_A(\mathcal{P}_i))$. 

Now, one can continue to form $\mathcal{P}$ from $\mathcal{P}'$ by adding covering relations between elements of $Ext(\mathcal{P})$. By Theorem~\ref{thm:h2pure}, the addition of each such covering relation increases the index by one. Thus, if a non-pure, height-two poset is Frobenius, then it must have the form of $\mathcal{P}'$.  We have established the following result. 

\begin{theorem}\label{thm:npurechar}
A height-two poset $\mathcal{P}$ is Frobenius if and only if it satisfies the following four conditions:

\begin{enumerate}[label=\textup(\roman*\textup)]
    \item the pure components of $\mathcal{P}$ are Frobenius, pure, height-two posets;
    \item there are no covering relations between maximal and minimal elements of a pure component of $\mathcal{P}$;
    \item each minimal element of $\mathcal{P}$ is covered by at most one maximal element of a given pure component of $\mathcal{P}$, and each maximal element covers at most one minimal element of a given pure component of $\mathcal{P}$;
    \item if the pure components of $\mathcal{P}$ are contracted to a point in the Hasse diagram of $\mathcal{P}$, then the result is a simple graph containing no cycles; that is, a tree.
\end{enumerate}
\end{theorem}

\begin{Ex}
The leftmost poset of Figure~\ref{ex:npurefrob} is an example of a Frobenius, non-pure, height-two poset.
\end{Ex}

\section{Rigidity}\label{sec:rigid}

In this section, we prove the rigidity result noted in the Introduction (see Theorem \ref{thm:main2}). The proof depends on the following result of Coll and Gerstenhaber, which itself is a corollary to their more general theorem regarding Lie semi-direct products for which type-A Lie poset algebras are the prime example.  To set the notation, let $\mathfrak{g}_A(\mathcal{P})$ be as above, $\mathfrak{h}$ be the standard Cartan subalgebra of $\mathfrak{g}$ with linear dual $\mathfrak{h}^*$, $\mathfrak{c}$ be the center of $\mathfrak{g}_A(\mathcal{P})$, and the $H^i$'s designate cohomology classes of Chevalley-Eilenberg or simplicial type, depending on whether the first argument is a Lie algebra or a simplicial complex, respectively.

\begin{theorem}[Coll and Gerstenhaber \textbf{\cite{CG}}, 2017] \label{CG}
$$H^2(\mathfrak{g}_A(\mathcal{P}),\mathfrak{g}_A(\mathcal{P}))=(\bigwedge^{~~~2}\mathfrak{h}^{*}\bigotimes\mathfrak{c})\quad\bigoplus\quad(\mathfrak{h}^{*}\bigotimes H^1(\Sigma(\mathcal{P}),\mathbf{k}))\quad\bigoplus\quad H^2(\Sigma(\mathcal{P}),\mathbf{k})$$
\end{theorem}

Observe that the necessary and sufficient conditions for a Lie poset algebra to be absolutely rigid, i.e., to have no infinitesimal deformations, is the simultaneous vanishing of $(\bigwedge^2\mathfrak{h}^*\bigotimes\mathfrak{c})$, $(\mathfrak{h}^*\bigotimes H^1(\Sigma(\mathcal{P}),\mathbf{k}))$, and $H^2(\Sigma(\mathcal{P}),\mathbf{k})$. 
Since we are only considering Frobenius Lie algebras, $\mathfrak{c}$ is trivial. To show that $(\mathfrak{h}^*\bigotimes H^1(\Sigma(\mathcal{P}),\mathbf{k}))$ and $H^2(\Sigma(\mathcal{P}),\mathbf{k})$ are also trivial, we invoke the  Universal Coefficient Theorem, where it suffices to show that $H_n(\Sigma(\mathcal{P}),\mathbf{k})=0$ for $n=1,2$. In fact, we prove a stronger result.

\begin{theorem}\label{Nohomology}
If $\mathcal{P}$ is a Frobenius poset of height two or less, then $\Sigma(\mathcal{P})$ is contractible.
\end{theorem}

For heights zero and one, Theorem~\ref{Nohomology} follows directly from Theorem~\ref{thm:h01FC} and the fact that, for these heights, the Hasse diagram is homotopic to the corresponding simplicial complex. The proof for height-two posets is less straightforward. In the pure case, we make use of the theory of discrete Morse functions \textbf{\cite{Forman},} which requires the following definitions and theorem. 
\\*

Let $\Sigma$ be a simplicial complex, and 
$\alpha^{(p)}\in \Sigma$ be a $p$-simplex.

\begin{definition}
A function $f:\Sigma\to \mathbb{R}$ is a discrete Morse function if for every $\alpha^{(p)}\in \Sigma$
$$|\{\beta^{(p+1)}\supset\alpha^{(p)}~|~\beta^{(p+1)}\in \Sigma, f(\beta^{(p+1)})\le f(\alpha^{(p)})\}|\le 1$$
and 
$$|\{\gamma^{(p-1)}\subset\alpha^{(p)}~|~\gamma^{(p-1)}\in \Sigma, f(\gamma^{(p-1)})\ge f(\alpha^{(p)})\}|\le 1.$$
\end{definition}

\begin{definition}
A simplex $\alpha^{(p)}$ is critical if $$|\{\beta^{(p+1)}\supset\alpha^{(p)}~|~\beta^{(p+1)}\in \Sigma, f(\beta^{(p+1)})\le f(\alpha^{(p)})\}|=0$$
and 
$$|\{\gamma^{(p-1)}\subset\alpha^{(p)}~|~\gamma^{(p-1)}\in \Sigma, f(\gamma^{(p-1)})\ge f(\alpha^{(p)})\}|=0.$$
\end{definition}

\begin{Ex}\label{ex:DMF}
Consider the simplicial complex $\Sigma$ illustrated in Figure~\ref{fig:Simplex}. A discrete Morse function with a single critical simplex of $v_1$ is obtained by assigning values as follows: $f(v_1)=0$, $f(e_1)=1$, $f(v_2)=2$, $f(e_2)=3$, $f(v_3)=4$, $f(e_3)=5$, $f(v_4)=6$, $f(f_1)=7$, $f(e_4)=8$, $f(f_2)=9$, and $f(e_5)=10$.
\end{Ex}

\begin{theorem}\label{thm:DMT}
Suppose $\Sigma$ is a simplicial complex with a discrete Morse function. Then $\Sigma$ is homotopy equivalent to a CW complex with exactly one cell of dimension $p$ for each critical simplex of dimension $p$.
\end{theorem}

We are now in a position to return to the proof of Theorem~\ref{Nohomology}.

\begin{proof}[Proof of Theorem~\ref{Nohomology}]
Recall from Section~\ref{sec:frobchar}, that all Frobenius, pure, height-two posets can be inductively constructed by gluing together (identifying minimal elements and maximal elements) copies of $\mathcal{P}(2,1,1)$ and $\mathcal{P}(1,1,2)$. So let $\mathcal{P}_n$ be a Frobenius, pure, height-two poset containing $n$ rank-one elements. The proof is by induction on the number of rank-one elements $n$.

For the base case consider $\Sigma(\mathcal{P}_1)$, which is homotopic to $\Sigma(\mathcal{P}(1,1,2))$ as well as $\Sigma(\mathcal{P}(2,1,1))$; that is, $\Sigma(\mathcal{P}_1)\cong\Sigma(\mathcal{P}(2,1,1))\cong\Sigma(\mathcal{P}(1,1,2))$. See Figure~\ref{fig:Simplex}.
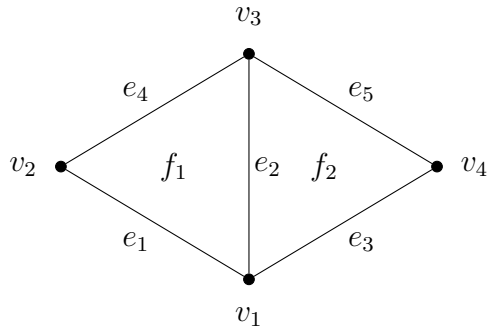
\begin{figure}[H]
$$\begin{tikzpicture}
\node (v1) at (-3,0) [circle, draw = black, fill = black, inner sep = 0.5mm] {};
\node (v4) at (-0.5,-1.5) [circle, draw = black, fill = black, inner sep = 0.5mm] {};
\node (v2) at (-0.5,1.5) [circle, draw = black, fill = black, inner sep = 0.5mm] {};
\node (v3) at (2,0) [circle, draw = black, fill = black, inner sep = 0.5mm] {};
\draw (v1) -- (v2) -- (v3) -- (v4) -- (v1);
\draw (v2) -- (v4);
\node at (-1.5,0) {$f_1$};
\node at (0.5,0) {$f_2$};
\node at (-0.5,2) {$v_3$};
\node at (1,1) {$e_5$};
\node at (2.5,0) {$v_4$};
\node at (1,-1) {$e_3$};
\node at (-0.5,-2) {$v_1$};
\node at (-2,-1) {$e_1$};
\node at (-3.5,0) {$v_2$};
\node at (-2,1) {$e_4$};
\node at (-0.25,0) {$e_2$};
\end{tikzpicture}$$
\caption{$\Sigma(\mathcal{P}(2,1,1))\cong\Sigma(\mathcal{P}(1,1,2))$}\label{fig:Simplex}
\end{figure}

\noindent
In Example~\ref{ex:DMF}, a discrete Morse function is given for $\Sigma(\mathcal{P}_1)$ with a single critical simplex of vertex $v_1$. Thus, by Theorem~\ref{thm:DMT}, $\Sigma(\mathcal{P}_1)$ is contractible.

Now, continue to adjoin copies of $\mathcal{P}(1,1,2)$ and $\mathcal{P}(2,1,1)$ to $\mathcal{P}_1$ to form the Frobenius poset, $\mathcal{P}_{n-1}$. Assume that there exists a discrete Morse function  $f:\Sigma(\mathcal{P}_{n-1})\to \mathbb{R}$ which has a single critical simplex corresponding to $v_1$ of $\Sigma(\mathcal{P}_1)$. Thus, $\Sigma(\mathcal{P}_{n-1})$ is contractible by Theorem~\ref{thm:DMT}. Now, form $\mathcal{P}_n$ from $\mathcal{P}_{n-1}$ by adjoining a copy of $\mathcal{P}(2,1,1)$ or $\mathcal{P}(1,1,2)$ in such a way that $\mathcal{P}_n$ is Frobenius; denote the simplicial complex corresponding to this new copy of $\mathcal{P}(2,1,1)$ or $\mathcal{P}(1,1,2)$ by $K$. Extending $f:\Sigma(\mathcal{P}_{n})\to \mathbb{R}$ breaks into three cases:
\begin{enumerate}
    \item If $\mathcal{P}_n$ is formed by attaching a copy of $\mathcal{P}(2,1,1)$ or $\mathcal{P}(1,1,2)$ to $\mathcal{P}_{n-1}$ via rules $A_1$, $A_2$, or $C$ of Lemma~\ref{lem:table}, then either vertex $v_1$ or $v_2$ in Figure~\ref{fig:Simplex} of $K$ is being identified with a vertex of $\Sigma(\mathcal{P}_{n-1})$. If the vertex is $v_1$, then extend $f$ so that $f(e_1)=p+1$, $f(v_2)=p+2$, $f(e_2)=p+3$, $f(v_3)=p+4$, $f(e_3)=p+5$, $f(v_4)=p+6$, $f(f_1)=p+7$, $f(e_4)=p+8$, $f(f_2)=p+9$, and $f(e_5)=p+10$. If the vertex is $v_2$, then extend $f$ so that $f(e_1)=p+1$, $f_(v_1)=p+2$, $f(e_4)=p+3$, $f(v_3)=p+4$, $f(f_1)=p+5$, $f(e_2)=p+6$, $f(e_5)=p+7$, $f(v_4)=p+8$, $f(f_2)=p+9$, $f(e_3)=p+10$.
    \item If $\mathcal{P}_n$ is formed by attaching a copy of $\mathcal{P}(2,1,1)$ of $\mathcal{P}(1,1,2)$ to $\mathcal{P}_{n-1}$ via rules $D_1$ or $D_2$, then $K$ is adjoined to $\Sigma(\mathcal{P}_{n-1})$ by identifying two edges each of which contain a vertex of degree two. Without loss of generality, assume we are identifying edge $e_1$ in Figure~\ref{fig:Simplex} of $K$. In this case, extend $f$ so that $f(e_4)=p+1$, $f(v_2)=p+2$, $f(f_1)=p+3$, $f(e_2)=p+4$, $f(e_3)=p+5$, $f(v_4)=p+6$, $f(f_2)=p+7$, and $f(e_5)=p+8$.
    \item If $\mathcal{P}_n$ is formed by attaching a copy of $\mathcal{P}(2,1,1)$ or $\mathcal{P}(1,1,2)$ to $\mathcal{P}_{n-1}$ via rule $F$, then adjacent edges of $K$ are identified with adjacent edges of $\Sigma(\mathcal{P}_{n-1})$, where both edges must contain a vertex of degree two; that is, either edges $e_1$ and $e_3$ in Figure~\ref{fig:Simplex} of $K$, or $e_1$ and $e_4$ in Figure~\ref{fig:Simplex} of $K$. If we are identifying edges $e_1$ and $e_3$ of $K$, then extend $f$ by $f(e_2)=p+1$, $f(v_3)=p+2$, $f(f_1)=p+3$, $f(e_4)=p+4$, $f(f_2)=p+5$, and $f(e_5)=p+6$. Otherwise, if we are identifying edges $e_1$ and $e_4$ of $K$, then extend $f$ by $f(f_1)=p+1$, $f(e_2)=p+2$, $f(e_5)=p+3$, $f(v_4)=p+4$, $f(f_2)=p+5$, and $f(e_3)=p+6$.
\end{enumerate} 
It is routine to verify that the resulting $f:\Sigma(\mathcal{P}_n)\to\mathbb{R}$ is a discrete Morse function, and that no new critical simplices have been added in extending of $f$. Thus, $f$ has a single critical simplex in vertex $v_1$ of $\Sigma(\mathcal{P}_1)$; that is, $\Sigma(\mathcal{P}_n)$ is contractible by Theorem~\ref{thm:DMT}. Therefore, the result for  Frobenius, pure, height-two posets follows by induction.

Finally, we consider non-pure, height-two posets. Let $\mathcal{P}$ be such a poset with pure components $\{\mathcal{P}_i\}_{i=1}^n$. By Theorem~\ref{thm:npurechar}, $\Sigma(\mathcal{P})$ has the property that, if the simplicial complexes corresponding to the $\{\mathcal{P}_i\}_{i=1}^n$ are contracted to a point, then the resulting simplicial complex is a tree. Since the simplicial complexes corresponding to the $\{\mathcal{P}_i\}_{i=1}^n$ are contractible, as are trees, the result follows.
\end{proof}

\begin{remark}
It is possible to extend the discrete Morse function above to include using rules $B$, $E_1$, $E_2$, $G_1$, $G_2$, and $H$ of Table~\ref{tab:h2fassem} in such a way that $B$, $E_1$, $E_2$, $G_1$, and $G_2$ contribute a single critical edge and $H$ contributes two critical edges. Thus, if $\mathcal{P}$ is a pure, height-two poset built from copies of $\mathcal{P}(1,1,2)$ or $\mathcal{P}(2,1,1)$ by applying the gluing rules of Table~\ref{tab:h2fassem}, then there exists $d\in\mathbb{Z}_{\ge 0}$ such that $\Sigma(\mathcal{P})$ is homotopic to a wedge product of $d$ one-spheres and $\ind(\mathfrak{g}_A(\mathcal{P}))=d$. Such a topological interpretation of the index also holds for connected, height-one posets; that is, if $\mathcal{P}$ is a connected, height-one poset, then $\Sigma(\mathcal{P})$ is a wedge product of $d$ one-spheres and $\ind(\mathfrak{g}_A(\mathcal{P}))=d$.
\end{remark}

\begin{remark}
In height-three, the natural analogue of the Frobenius posets $\mathcal{P}(1,1,2)$ and $\mathcal{P}(2,1,1)$ is the poset $\mathcal{Q}=\{1,2,3,4,5,6\}$ defined by the relations $1\preceq 2\preceq 3,4$; $3\preceq 5$; and $4\preceq 6$, along with its dual $\mathcal{Q}^*$. See Figure~\ref{fig:tuningfork}.
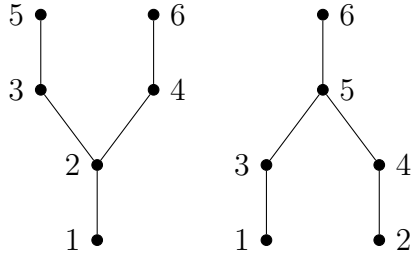
\begin{figure}[H]
$$\begin{tikzpicture}
\node (v4) at (-2,3) [circle, draw = black, fill = black, inner sep = 0.5mm, label=left:{5}] {};
\node (v6) at (-0.5,3) [circle, draw = black, fill = black, inner sep = 0.5mm, label=right:{6}] {};
\node (v5) at (-0.5,2) [circle, draw = black, fill = black, inner sep = 0.5mm, label=right:{4}] {};
\node (v3) at (-2,2) [circle, draw = black, fill = black, inner sep = 0.5mm, label=left:{3}] {};
\node (v2) at (-1.25,1) [circle, draw = black, fill = black, inner sep = 0.5mm, label=left:{2}] {};
\node (v1) at (-1.25,0) [circle, draw = black, fill = black, inner sep = 0.5mm, label=left:{1}] {};
\draw (v1) -- (v2) -- (v3) -- (v4);
\draw (v2) -- (v5) -- (v6);
\node (v7) at (1,0)  [circle, draw = black, fill = black, inner sep = 0.5mm, label=left:{1}] {};
\node (v11) at (2.5,0)  [circle, draw = black, fill = black, inner sep = 0.5mm, label=right:{2}] {};
\node (v12) at (2.5,1)  [circle, draw = black, fill = black, inner sep = 0.5mm, label=right:{4}] {};
\node (v8) at (1,1)  [circle, draw = black, fill = black, inner sep = 0.5mm, label=left:{3}] {};
\node (v9) at (1.75,2)  [circle, draw = black, fill = black, inner sep = 0.5mm, label=right:{5}] {};
\node (v10) at (1.75,3)  [circle, draw = black, fill = black, inner sep = 0.5mm, label=right:{6}] {};
\draw (v7) -- (v8) -- (v9) -- (v10);
\draw (v11) -- (v12) -- (v9);
\end{tikzpicture}$$
\caption{$\mathcal{Q}$ and $\mathcal{Q^*}$}\label{fig:tuningfork}
\end{figure}
\noindent
Both $\mathcal{Q}$ and $\mathcal{Q}^*$ are Frobenius posets. Moreover, it can be shown that one obtains Frobenius height-three posets by gluing together copies of $\mathcal{Q}$ and $\mathcal{Q}^*$ using the height-three analogues of gluing rules $A_1$, $A_2$, $B$, $D_1$, $D_2$, and $F$. Interestingly, we once again have that if $\mathcal{P}$ is a connected poset built from copies of $\mathcal{Q}$ or $\mathcal{Q}^*$ using the height-three analgoes of the gluing rules in Table~\ref{tab:h2fassem}, then there exists $d\in\mathbb{Z}_{\ge 0}$ such that $\Sigma(\mathcal{P})$ is homotopic to a wedge product of $d$ one-spheres and $\ind(\mathfrak{g}_A(\mathcal{P}))=d$.  Of more important note is that height-three Frobenius posets are not completely characterized as in Lemma \ref{lem:table}. For example, the Frobenius poset $\mathcal{P}(1,2,2,2)$ cannot be built using this analogous gluing procedure.  See Figure \ref{stargates1} \textup(left\textup).
\end{remark}

\begin{remark} Calculations suggest that Theorem \ref{Nohomology} is true for posets of arbitrary height, and we conjecture that this is so.
The following examples are suggestive. It can be shown that the type-A Lie poset algebras associated with $\mathcal{P}(1,2,\hdots,2)$ and $\mathcal{P}(2,\hdots,2,1)$ as well as the natural generalization of our running example $\mathcal{P}(1,1,2)$, which we denote by $SG(n)$ and define by $1\preceq\hdots\preceq n$ and $1\preceq\hdots\preceq\lceil\frac{n}{2}\rceil\preceq n+1$, are Frobenius.  See Figure \ref{stargates1}.

\begin{figure}[H]
$$\begin{tikzpicture}
	\node (1) at (0, 0) [circle, draw = black, fill = black, inner sep = 0.5mm]{};
	\node (2) at (-0.5, 0.5)[circle, draw = black, fill = black, inner sep = 0.5mm] {};
	\node (3) at (0.5, 0.5) [circle, draw = black, fill = black, inner sep = 0.5mm] {};
    \node (4) at (-0.5, 1) [circle, draw = black, fill = black, inner sep = 0.5mm] {};
    \node (5) at (0.5, 1) [circle, draw = black, fill = black, inner sep = 0.5mm] {};
    \node (6) at (0, 1.5) {$\vdots$};
    \node (7) at (-0.5, 2)[circle, draw = black, fill = black, inner sep = 0.5mm] {};
	\node (8) at (0.5, 2) [circle, draw = black, fill = black, inner sep = 0.5mm] {};
    \node (9) at (-0.5, 2.5) [circle, draw = black, fill = black, inner sep = 0.5mm] {};
    \node (10) at (0.5, 2.5) [circle, draw = black, fill = black, inner sep = 0.5mm] {};
    \draw (4)--(3)--(1)--(2)--(5);
    \draw (5)--(3);
    \draw (2)--(4);
    \draw (7)--(9)--(8);
    \draw (7)--(10)--(8);
    \addvmargin{1mm}
\end{tikzpicture}\quad\quad\begin{tikzpicture}
\node (v1) at (-0.5,0) [circle, draw = black, fill = black, inner sep = 0.5mm] {};
\node at (-0.5,0.5) [circle, draw = black, fill = black, inner sep = 0.5mm] {};
\node (v3) at (-0.5,1) [circle, draw = black, fill = black, inner sep = 0.5mm] {};
\node at (-0.5,1.5) [circle, draw = black, fill = black, inner sep = 0.5mm] {};
\node (v2) at (-0.5,2) [circle, draw = black, fill = black, inner sep = 0.5mm] {};
\node (v4) at (0,1.5) [circle, draw = black, fill = black, inner sep = 0.5mm] {};
\node (v5) at (-0.5, 2.5) [circle, draw = black, fill = black, inner sep = 0.5mm] {};
\draw (v1) -- (v2);
\draw (v3) -- (v4);
\draw (v2) -- (v5);
\end{tikzpicture}$$
\caption{$\mathcal{P}(1,2,\hdots,2)$ and $SG(6)$}\label{stargates1}
\end{figure}
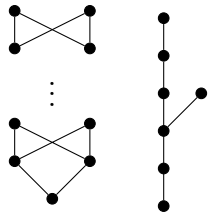

\noindent
For $\mathcal{P}=\mathcal{P}(1,2,\hdots,2)$ or $\mathcal{P}(2,\hdots,2,1)$, one finds that $H_n(\Sigma(\mathcal{P}),\mathbf{k})=0$ for $n>0$ by using the Mayor-Vietoris Sequence and the fact that $\Sigma(\mathcal{P})$ is obtained by inductively taking two-point suspensions starting from a point. As for $SG(n)$, it is clear that the simplicial complex $\Sigma(SG(n))$ is formed by adjoining a $\lceil\frac{n}{2}\rceil$-simplex to a $(n-1)$-simplex along a face. Such a space is star-convex and thus contractible.  Furthermore, taking any number of the Frobenius posets mentioned above, it follows from Remark~\ref{rem:minmax} that identifying a single minimal (resp. maximal) element of each also results in a Frobenius poset.  At the simplical level this corresponds to taking a wedge sum, so that the Mayer-Veotoris Theorem, once again, gives contractablity. 
\end{remark}

We have the following immediate corollary to Theorem~\ref{Nohomology}.

\begin{corollary}
If $\mathcal{P}$ be a Frobenius poset of height two or less, then

$$H^2(\Sigma(\mathcal{P}),\mathbf{k})=H^1(\Sigma(\mathcal{P}),\mathbf{k})=0.$$

\end{corollary}

\noindent
Upon applying Theorem~\ref{CG}, we have thus established the rigidity theorem noted in the Introduction.

\begin{theorem}\label{thm:main2}
A Frobenius Lie poset subalgebra of $\mathfrak{sl}(n)$ corresponding to a poset of height zero, one, or two is absolutely rigid.
\end{theorem}

\begin{remark}  If $\mathfrak{g}_A(\mathcal{P})$ is Frobenius, and $\mathcal{P}$ is of height two or less, then by a now-classical theorem of Gerstenhaber and Schack \textup{\textbf{\cite{G3}}}, the second Hochschild cohomology group $H^2(A(\mathcal{P}),A(\mathcal{P}))$ is trivial. This implies that the associative poset algebras corresponding to such Frobenius posets are also rigid.  
\end{remark}

\noindent

\section{Lie poset algebras in types B, C, and D}\label{sec:tbcd}

In this section, we provide definitions for posets of types B, C, and D, which allow us to develop matrix representations of Lie poset algebras in the other classical types.  The treatment here is consistent with the type-A approach; that is, such posets are in bijective correspondence with subalgbras which lie between a Cartan and Borel subalgebra. The proofs that these representations are well-defined are routine and are omitted.

\begin{definition}\label{def:bcd}
A Type C poset is a poset, \textup($\mathcal{P}, \preceq_{\mathcal{P}}\textup)$, on $\{-n,\hdots-1,1,\hdots, n\}$ such that
\begin{enumerate}
	\item If $i\preceq_{\mathcal{P}}j$, then $i\preceq_{\mathbb{Z}}j$;
    \item if $i\preceq_{\mathcal{P}}j$ and $j\preceq_{\mathcal{P}}k$, then $i\preceq_{\mathcal{P}}k$;
	\item if $i\neq -j$, then $i\preceq_{\mathcal{P}}j$ if and only if $-j\preceq_{\mathcal{P}}-i$.
\end{enumerate}
A Type B/D poset is a poset, \textup($\mathcal{P}, \preceq_{\mathcal{P}}\textup)$, on $\{-n,\hdots-1,1,\hdots, n\}$ satisfying 1-3 above as well as 
\begin{enumerate}
    \setcounter{enumi}{3}
    \item If $i\preceq_{\mathcal{P}}j$, then $-j\npreceq_{\mathcal{P}} i$.
\end{enumerate}
\end{definition}
\noindent
To find matrix representations in type-C and type-D, we now label the rows and columns by $\{-n,\hdots,-1,1\hdots, n\}$. For type B, an extra row and column must be added which intersect at the top left entry which contains the only nonzero entry of each, a one.

\begin{Ex}\label{ex:typeBCD}
The poset $\mathcal{P}$ on $\{-3, -2, -1, 1, 2, 3\}$ defined by $-1\preceq 2,3$; $-2\preceq 1,3$; and $-3\preceq 1,2$ 
may be regarded as a poset in types B, C, and D. The Hasse diagram of $\mathcal{P}$ is illustrated in Figure~\ref{fig:tBCD} (left).  The matrix representations of 
 $\mathfrak{g}_C(\mathcal{P})$ and $\mathfrak{g}_D(\mathcal{P})$\textup) are illustrated in Figure~\ref{fig:tBCD} (right).  
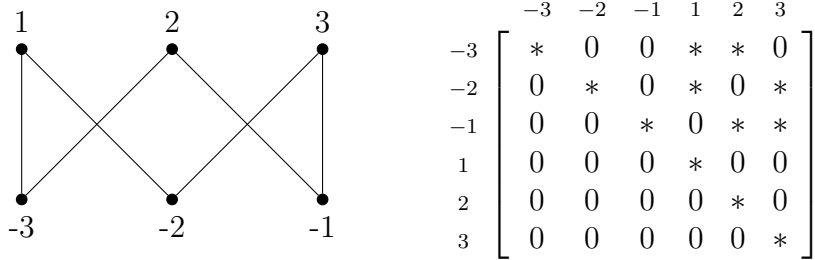
\begin{figure}[H]
$$\begin{tikzpicture}
	\node (-3) at (0, 0) [circle, draw = black, fill = black, inner sep = 0.5mm, label=below:{-3}]{};
	\node (-2) at (2, 0)[circle, draw = black, fill = black, inner sep = 0.5mm, label=below:{-2}] {};
	\node (-1) at (4, 0) [circle, draw = black, fill = black, inner sep = 0.5mm, label=below:{-1}] {};
	\node (1) at (0, 2) [circle, draw = black, fill = black, inner sep = 0.5mm, label=above:{1}]{};
	\node (2) at (2, 2)[circle, draw = black, fill = black, inner sep = 0.5mm, label=above:{2}] {};
	\node (3) at (4, 2) [circle, draw = black, fill = black, inner sep = 0.5mm, label=above:{3}] {};
	\node (5) at (7,1) {
  \kbordermatrix{
    & -3 & -2 & -1 & 1 & 2 & 3 \\
   -3 & * & 0 & 0 & * & * & 0  \\
   -2 & 0 & * & 0 & * & 0 & * \\
   -1 & 0 & 0 & * & 0 & * & * \\
   1 & 0 & 0 & 0 & * & 0 & 0 \\
   2 & 0 & 0 & 0 & 0 & * & 0 \\
   3 & 0 & 0 & 0 & 0 & 0 & * \\
  }
};
    \draw (1)--(-3)--(2);
    \draw (1)--(-2)--(3);
    \draw (2)--(-1)--(3);
    \addvmargin{1mm}
\end{tikzpicture}$$
\caption{Hasse diagram of $\mathcal{P}$ (left) and the matrix algebra  $\mathfrak{g}_C(\mathcal{P})$ (right)}\label{fig:tBCD}
\end{figure}
\end{Ex}

\begin{remark}
Theorem~\ref{CG} carries over mutatis mutandis to Lie poset algebras of types B, C, and D since such Lie poset algebras are also Lie semi-direct products.  
\end{remark}

\begin{remark}
In \textup{\textbf{\cite{Mayers}}}, Mayers has developed combinatorial index formulas for Lie poset algebras in types B, C, and D with certain height restrictions. From this, one can show that the poset $\mathcal{P}$ of Example~\ref{ex:typeBCD} corresponds to a Frobenius Lie poset algebra of types B, C, and D, but not type A. Note that $\Sigma(\mathcal{P})$ is homeomorphic to $S^1$, so $\Sigma(\mathcal{P})$ is not contractible. Furthermore, the corresponding Lie algebra is also not rigid; this resolves a question of Gerstenhaber and Giaquinto about the existence of such algebras \textup{\textbf{\cite{Prin}}}. Thus, contractability and rigidity seem to be necessary conditions for a poset to be Frobenius only in type A.
\end{remark}

\section{Epilogue}

As noted in the footnote in the Introduction,  Frobenius Lie algebras have important applications in physics. Recent work, by the current authors and others, has focused the search for Frobenius Lie algebras on the evocatively-named seaweed algebras introduced by Dergachev and A. Kirillov \textbf{\cite{DK}}.  Here, we have broadened the ``Frobenius search" inside the classical Lie algebras to include Lie poset algebras.  

Of course, to determine if a given Lie algebra is Frobenius requires a combinatorial mechanism for the computation of the index. For seaweed algebras, a successful approach has been to introduce the meander graph of a seaweed \textbf{\cite{Cameron, Coll3, Coll1, CHM, Coll2, DK, Elash,Panyushev2, Panyushev3}}.  The meander of a seaweed is an associated planar graph and the index of the seaweed can be computed by then counting the number and type of connected components of the meander \textbf{\cite{CHM,DK}}.\footnote{From these combinatorial index formulas all closed-form general closed-form index formulas where the index is given by a polynomial greatest common divisor formula in the sizes of the parts of the composition that define the seaweed, have recently been catalogued \textbf{\cite{Cameron, Coll2, Kar}}.} 

A prime motivation for this article is the parallel development of the requisite combinatorics to deliver closed-form index formulas for the index of Lie poset algebras.  These formulas result from an analysis of the chains of the poset defining the Lie poset algebra. A secondary motivation for our investigation is the observation that seaweed algebras and Lie poset algebras maintain similar \textit{spectral} properties which we now describe.

Let $\mathfrak{g}$ be a Frobenius Lie algebra with $F$ an associated index-realizing functional.\footnote{Index realizing, or \textit{regular}, functionals exist in abundance and are dense in $\mathfrak{g}$ and $\mathfrak{g^*}$ in both the Euclidean and Zariski topologies.} That is, $$\dim(\ker(F))=\ind(\mathfrak{g})=0.$$ In this case, the natural map $\mathfrak{g} \rightarrow \mathfrak{g}^*$ defined by $x \mapsto  F[x,-]$ is an isomorphism.  The image of $F$ under the inverse of this map is called a \textit{principal element} of $\mathfrak{g}$ and will be denoted $\widehat{F}$.  It is the unique element of $\mathfrak{g}$ such that 

$$
F\circ \widehat{F}= F([\widehat{F},-]) = F.  
$$
In \textbf{\cite{Ooms}}, Ooms established that the spectrum of the adjoint of a principal element of a Frobenius Lie algebra is independent of the principal element chosen to compute it (see also \textbf{\cite{G2}}, Theorem 3).  Generally, the eigenvalues of ad~$\widehat{F}$ can take on virtually any value (see \textbf{\cite{Diatta}} for examples). But, in their formal study of principal elements \textbf{\cite{G3}}, Gerstenhaber and Giaquinto showed that if $\mathfrak{g}$ is a Frobenius seaweed subalgebra of $\mathfrak{sl}(n)$, then the spectrum of the adjoint of a principal element of $\mathfrak{g}$ consists entirely of integers.\footnote{Joseph, seemingly unaware of the Type-A result of \textbf{{\cite{G3}}}, but using different methods, strongly extended this integrality result to all seaweed subalgebras of semisimple Lie algebras
\textbf{\cite{Joseph}.}} Subsequently, Coll et al. \textbf{\cite{unbroken}} 
showed that this spectrum must actually be an \textit{unbroken} sequence of integers centered at one half. Moreover, the dimensions of the associated eigenspaces are shown to form a symmetric distribution.  This is true more generally and we have the following theorem.

\begin{theorem}[Coll, et al. \textbf{\cite{specD,specAB,unbroken}}]\label{thm:main}
If $\mathfrak{g}$ is a Frobenius seaweed subalgebra of classical type, and $\widehat{F}$ is a principal element of $\mathfrak{g}$, then the spectrum of $\text{ad }\widehat{F}$ consists of an unbroken set of integers centered at one-half.  Moreover, the dimensions of the associated eigenspaces form a symmetric distribution. 
\end{theorem}


Remarkably, Theorem \ref{thm:main} is true for type-A Frobenius Lie poset algebras corresponding to posets of restricted height -- but the spectrum is 
 ``narrow", in the sense of the following Theorem.

\begin{theorem}[Mayers \textbf{\cite{Mayers}}]\label{thm:pspec}
If $\mathfrak{g}_A(\mathcal{P})$ is Frobenius, where $\mathcal{P}$ is of height two or less, then the spectrum of $\mathfrak{g}_A(\mathcal{P})$ consists of an equal number of 0's and 1's.
\end{theorem}


\begin{remark}  The unbroken, symmetric spectrum result of Theorem \ref{thm:main} does not characterize seaweeds. For example, consider the Frobenius poset $\mathcal{P}(2,1,1)$.  Note that $\mathfrak{g}_A(\mathcal{P}(2,1,1))$ has rank three, dimension eight, principal element $\widehat{F}=\rm{diag}\left(\frac{1}{2},\frac{1}{2},-\frac{1}{2},-\frac{1}{2}\right)$, and spectrum given by the multiset $\{0,0,0,0,1,1,1,1\}$.  It is not a seaweed \textup(see \textup{\textbf{\cite{unbroken}}}, Example 18\textup).  We ask the following questions:
\begin{enumerate}[label=\textup(\roman*\textup)]
    \item Is Theorem \ref{thm:pspec} true in type A for heights greater than two?  What about types B, C, and D?
    \item Does the narrow spectrum of a Frobenius subalgebra of classical type characterize Lie poset algebras?
    \item Can the unbroken spectrum of a Lie subalgebra of classical type ``evolve" under the deformation of the algebra.  Note that by a recent result of Elashvilli \textup{\textbf{\cite{Elash2}}}, this question is moot for type-A seaweeds -- since they are cohomologically inert. What about types B, C, and D?
\item What is it about the geometry of the underlying algebraic group that accounts for the unbroken spectrum?

\end{enumerate}

\end{remark}

\section{Appendix A - a matrix reduction algorithm}
In this section, we provide an algorithm for transforming $C(\mathfrak{g}(\mathcal{P}))$ into the equivalent matrix $C'(\mathfrak{g}(\mathcal{P}))$ for a poset algebra $\mathfrak{g}(\mathcal{P})$ corresponding to a connected poset $\mathcal{P}$. To describe the algorithm, it is necessary to partition the non-minimal elements of $\mathcal{P}$ into subsets $M_j$ for each minimal element $j\in \mathcal{P}$. Define $M_j$ for each minimal element $j\in \mathcal{P}$ by $m_j\in M_j$ if and only if $j$ is the least minimal element, with respect to the natural ordering on $\mathbb{Z}$, of $\mathcal{P}$ satisfying $j\preceq_{\mathcal{P}}m_j$. Throughout this section, we will assume that the minimal elements of $\mathcal{P}$, say $\{1,\hdots,n\}$, satisfy $j\preceq_{\mathcal{P}} m$ for some $m\in M_{j-1}$ for all $j=1,\hdots,n$. Note that this can always be arranged since, as stated in the preliminaries, $\mathfrak{g}(\mathcal{P})$ is invariant under the choice of linear extension of $\mathcal{P}$.

First, order the rows of $C'(\mathfrak{g}(\mathcal{P}))$ according to the following rubric:   
\begin{enumerate}[label={(\bfseries R\arabic*)}]
    \item $\sum \mathbf{E_{i,i}}$ followed by the rows of the form $\mathbf{E_{i,m_j}}$ for $i>1$ minimal, $m_j\in M_j$ maximal for $i\neq j$ such that $m_j$ is not maximal in $\mathbb{Z}$ with this property in the lexicographic ordering of the subscripts $(i,m_j)$ in $\mathbb{Z}\times\mathbb{Z}$;
    \item $\mathbf{E_{i,m_j}}$ for $i>1$ minimal, $m_j\notin M_i$ maximal and $m_j$ maximal with this property in $\mathbb{Z}$ listed in increasing order of $i$ in $\mathbb{Z}$;
    \item $\mathbf{E_{i,m_i}}$ for $i$ minimal and $m_i\in M_i$ in increasing order of $m_i$ in $\mathbb{Z}$;
    \item $\mathbf{E_{m_i,m_i}}$ for $m_i\in M_i$ maximal in increasing order of $m_i$ in $\mathbb{Z}$;
    \item $\mathbf{E_{i,i}}$ for $i>1$ minimal in $\mathcal{P}$ listed in increasing order of $i$ in $\mathbb{Z}$;
    \item finally, rows corresponding to $\mathbf{E_{i,m_j}}$, $\mathbf{E_{m_j,m_j}}$, and $\mathbf{E_{m_j,k}}$ for $i>1$ minimal in $\mathcal{P}$, $m_j\notin M_i$, $m_j\notin Ext(\mathcal{P})$ listed so that the subscripts are in increasing lexicographic order in $\mathbb{Z}\times\mathbb{Z}$ for each fixed $m_j$ and so that these groups occur in increasing order of $m_j$ in $\mathbb{Z}$.
\end{enumerate}
Now, order the columns as follows
\begin{enumerate}[label={(\bfseries C\arabic*)}]
    \item $\sum \mathbf{E_{i,i}}$ followed by $\mathbf{E_{i,i}}$ listed from $i=2$ to $i=|\mathcal{P}|$;
    \item $\mathbf{E_{i,m_i}}$ for $i$ minimal, $m_i\in M_i$ maximal listed in increasing order of $m_i$ in $\mathbb{Z}$;
    \item $\mathbf{E_{i,m_j}}$ for $i$ minimal, $m_j\notin M_i$ maximal, and  $m_j$ minimal in $\mathbb{Z}$ with this property listed in increasing order of $i$ in $\mathbb{Z}$;
    \item $\mathbf{E_{i,m_k}}$ for $i$ minimal and $m_k$ maximal for $k\neq i$, excluding $k=m_j$ described in \textbf{(C3)}, listed in lexicographic order of the subscripts in $\mathbb{Z}\times\mathbb{Z}$;
    \item finally, $\mathbf{E_{i,m_j}}$ as well as $\mathbf{E_{m_j,k}}$ for $i$ minimal and $m_j\notin Ext(\mathcal{P})$ listed in increasing lexicographic ordering of the subscripts in $\mathbb{Z}\times\mathbb{Z}$ for each fixed $m_j$ in increasing order of $m_j$ in $\mathbb{Z}$.
\end{enumerate}

\noindent
Next, with the rows and columns ordered as described above, we will perform a sequence of row operations. Assuming, as stated above, that the minimal elements of $\mathcal{P}$ are $\{1,\hdots,n\}$, perform the following row operations working from $j=1$ up to $n$.
\begin{enumerate}
    \item perform $\mathbf{E_{i,m_j}}-\frac{E_{i,m_j}}{E_{j,m_j}}\mathbf{E_{j,m_j}}$ at row $\mathbf{E_{i,m_j}}$ for $i$ minimal and $m_j\notin M_i$ such that $m_j\in M_j$;
    \item perform $\mathbf{E_{j+1,k}}-\frac{E_{j+1,k}}{E_{j+1,m_t}}\mathbf{E_{j+1,m_t}}$ at row $\mathbf{E_{j+1,k}}$ for $k\neq m_t$ and $m_t\notin M_i$ such that $m_t\in M_t$ for $1\le t\le j$ is maximal in $\mathcal{P}$ and $m_t$ is maximal in $\mathbb{Z}$ with this property;
    \item multiply row $\mathbf{E_{j,m_j}}$ by $\frac{1}{E_{j,m_j}}$ for $m_j\in M_j$; and
    \item multiply row $\mathbf{E_{j,m_t}}$ by $-\frac{1}{E_{j,m_t}}$ for $m_t\notin M_j$ such that $m_t\in M_t$ for $1\le t\le j$ is maximal in $\mathcal{P}$ and $m_t$ maximal in $\mathbb{Z}$ with this property.
\end{enumerate}

\noindent
Finally, perform the following row operations
\begin{enumerate}
    \item perform $\mathbf{E_{m_i,m_j}}+\frac{E_{m_i,m_j}}{E_{i,m_i}}\mathbf{E_{i,m_i}}-\frac{E_{m_i,m_j}}{E_{j,m_j}}\mathbf{E_{j,m_j}}$ at row $\mathbf{E_{m_i,m_j}}$ for $m_i\in M_i$ and $m_j\in M_j$;
    \item multiply row $\mathbf{E_{i,i}}$ by $-\frac{1}{E_{i,i}}$ for $i$ maximal in $\mathcal{P}$;
    \item multiply row $\mathbf{E_{i,i}}$ by $\frac{1}{E_{i,m_j}}$ for $i$ minimal in $\mathcal{P}$ and $m_j\notin M_i$ is the maximal element of $\mathcal{P}$ which is minimal in $\mathbb{Z}$ with this property;
\end{enumerate}

\begin{remark}
Applying the above algorithm to transform $C(\mathfrak{g}(\mathcal{P}))$ into the equivalent matrix $C'(\mathfrak{g}(\mathcal{P}))$ we have the following
\begin{itemize}
    \item row $\mathbf{E_{i,k}}$ of $C'(\mathfrak{g}(\mathcal{P}))$ for $i>1$ minimal, $k\in \mathcal{P}$ maximal in $\mathcal{P}$, $k$ not maximal in $\mathbb{Z}$ satisfying $k\notin M_i$, and $i\preceq_{\mathcal{P}}k$ is a zero row.
    \item row $\mathbf{E_{i,m_j}}$ of $C'(\mathfrak{g}(\mathcal{P}))$ for $i>1$ minimal, $m_j\in \mathcal{P}$ maximal in $\mathbb{Z}$ satisfying $m_j\notin M_i$ and $i\preceq_{\mathcal{P}}m_j$ is the unique row with a nonzero entry in column $\mathbf{E_{i,i}}$;
    \item row $\mathbf{E_{i,m_i}}$ of $C'(\mathfrak{g}(\mathcal{P}))$ for $i$ minimal and $m_i\in M_i$ is the unique row with a nonzero entry in column $\mathbf{E_{m_i,m_i}}$;
    \item row $\mathbf{E_{m_i,m_i}}$ of $C'(\mathfrak{g}(\mathcal{P}))$ for $m_i\in M_i$ maximal in $\mathcal{P}$ is the unique row with a nonzero entry in column $\mathbf{E_{i,m_i}}$;
    \item row $\mathbf{E_{i,i}}$ of $C'(\mathfrak{g}(\mathcal{P}))$ for $1<i\in \mathcal{P}$ minimal is linearly independent from the rest as by connectivity there must exist $m_j$ for $i\neq j$ such that $i\preceq_{\mathcal{P}}m_j$, i.e., row $\mathbf{E_{i,i}}$ has a nonzero entry in column $\mathbf{E_{i,m_j}}$. The only other row with a nonzero entry in this column is $\mathbf{E_{m_j,m_j}}$ which is also the unique row with nonzero entry in column $\mathbf{E_{j,m_j}}$;
\end{itemize}
\end{remark}

\section{Appendix B - Index of $\mathcal{P}(n,1,m)$}

In this appendix, we develop index formulas for Lie posets algebras of the form $\mathfrak{g}_A(\mathcal{P}(n,1,m))$.  We will perform a standard ``squeeze-play" by finding an upper bound for the index, then a lower bound and then showing that the two match. For the upper bound, we make a judicious choice of functional $F$.  It follows from the original definition for the index of a Lie algebra $\mathfrak{g}$, i.e., $\min_{F\in \mathfrak{g}}\dim\ker(B_F)$, that an arbitrary $F\in\mathfrak{g}^*$ satisfies $\dim\ker(B_F)\ge\ind(\mathfrak{g})$. As for determining a lower bound, we make use of a relationship between matchings on graphs and the rank of skew-symmetric matrices which is descibed and utilized in Section~\ref{sec:lb}.

\subsection{Upper Bounds}\label{sec:ub}

In this subsection, we determine upper bounds on the index of $\mathfrak{g}_A(\mathcal{P}(n,1,m))$. Throughout this section let $E^*_{i,j}$ denote the functional which returns the $i,j$-entry of a matrix. The heuristic for upper bound proofs using functionals works as follows: given a functional $F$ on a Lie algebra $\mathfrak{g}$ with basis $\{x_1,\hdots,x_n\}$
\begin{enumerate}
    \item Let $B\in\mathfrak{g}\cap\ker(F)$;
    \item determine the restrictions $F([x_i,B])=0$ places on the entries of $B$ for each basis element $x_i$ of $\mathfrak{g}$;
    \item solve the resulting system of equations to determine $\dim\ker(B_F)\ge\ind(\mathfrak{g})$.
\end{enumerate} 
\noindent
We will work in $\mathfrak{gl}(n)$ to determine an upper bound on the index of $\mathfrak{g}(\mathcal{P})$ and then subtract one to determine the corresponding upper bound on the index of $\mathfrak{g}_A(\mathcal{P})$. Performing calculations in $\mathfrak{gl}(n)$ allows the use of the basis consisting of $E_{i,j}$ for $i,j\in\mathcal{P}$ and $i\preceq j$ as well as $E_{i,i}$ for $i\in\mathcal{P}$.

\begin{lemma}\label{lem:ubnn}
If $n\in\mathbb{Z}_{>0}$, then $\ind\mathfrak{g}_A(\mathcal{P}(n, 1, n))\le n^2-2n+2$.
\end{lemma}
\begin{proof}
Let $$F=\sum_{i=1}^nE^*_{1, n+1+i}+\sum_{i=2}^nE^*_{i, n+i}$$ and assume that $B\in\mathfrak{g}_A(\mathcal{P}(n,1,n))\cap\ker(B_F)$. The restrictions on the entries of $B$ imposed by basis elements of $\mathfrak{g}(\mathcal{P}(n,1,n))$ break into seven cases.

\bigskip
\noindent 
\textbf{Case 1:} $E_{1, n+1+i}$ for $1\le i\le n$ and $E_{j, n+j}$ for $1<j\le n$.
These basis elements contribute the conditions $E^*_{1,1}(B)=E^*_{n+1+i, n+1+i}(B)$ as well as $E^*_{j,j}(B)=E^*_{n+j, n+j}(B)$.

\bigskip
\noindent 
\textbf{Case 2:} $E_{i, n+1}$ for $1\le i<n+1$. These basis elements contribute the conditions $\sum_{i=1}^nE^*_{n+1, n+1+i}(B)=0$ as well as $E^*_{n+1, n+i}(B)=0$.

\bigskip
\noindent 
\textbf{Case 3:} $E_{n+1, n+1+i}$ for $1\le i\le n$. These basis elements contribute the conditions $E^*_{1, n+1}(B)=-E^*_{i+1, n+1}(B)$ as well as $E^*_{1, n+1}(B)=0$.

\bigskip
\noindent 
\textbf{Case 4:} $E_{1,1}$. This basis element forces the condition $\sum_{i=1}^nE^*_{1, n+1+i}(B)=0$.

\bigskip
\noindent 
 \textbf{Case 5:} $E_{i,i}$ for $1<i\le n+1$. For $i<n+1$ the basis elements force the condition $E^*_{i, n+i}(B)=0$ while $i=n+1$ contributes nothing.

\bigskip
\noindent 
 \textbf{Case 6:} $E_{i,i}$ for $n+1<i=n+j<2n+1$. These basis elements contribute the conditions $E^*_{1, n+j}(B)+E^*_{j, n+j}(B)=0$.

\bigskip
\noindent
\textbf{Case 7:} $E_{2n+1, 2n+1}$. This basis element forces the condition $E^*_{1, 2n+1}(B)=0$.
\bigskip

\noindent
Now, we find that Cases 5, 6 and part of Case 3 tell us that $E^*_{1, i}(B)=0$ for $n+1\le i\le 2n+1$ as well as $E^*_{i, n+i}(B)=0$ for $1<i\le n$. Case 3 all together gives $E^*_{i, n+1}(B)=0$ for $1\le i\le n$. Case 2 all together gives $E^*_{n+1, i}(B)=0$ for $n+1<i\le 2n+1$. Finally, Case 1 allows us to conclude that $E^*_{1,1}(B)=E^*_{i, i}(B)$ for $i\neq n+1$ and thus along with $E^*_{n+1, n+1}(B)$ we get a contribution of two to $\dim(\ker(B_F))$. So the question remaining is: How many entries of $B$ are unrestricted? These correspond to $E^*_{i, n+j}(B)$ for $1<j\neq i\le n$; that is, $(n-1)^2$ entries. Thus, $\dim(\ker(B_F))$ is equal to

$$(n-1)^2+2=n^2-2n+3=n(n-2)+3,$$ 

\noindent
which gives an upper bound on $\ind(\mathfrak{g}(\mathcal{P}(n,1,n)))$. The result follows by restricting to $\mathfrak{sl}(2n+1)$.
\end{proof}

\begin{lemma}\label{lem:ubmn}
If $n,m\in\mathbb{Z}_{>0}$ satisfying $m<n$, then $\ind\mathfrak{g}_A(\mathcal{P}(m, 1, n))\le n(m-2)$.
\end{lemma}
\begin{proof}
Let $$F=\sum_{i=1}^nE^*_{1, m+1+i}+\sum_{i=2}^{m+1}E^*_{i, m+i}$$ and assume $B\in\mathfrak{g}_A(\mathcal{P}(n,1,n))\cap\ker(B_F)$. The restrictions on the entries of $B$ imposed by basis elements of $\mathfrak{g}(\mathcal{P}(m,1,n))$ break into seven cases.

\bigskip
\noindent
\textbf{Case 1:} $E_{1, m+1+i}$ for $1\le i\le n$ and $\mathbf{E_{j, m+j}}$ for $1<j\le m+1$. These basis elements contribute the conditions $E^*_{1,1}(B)=E^*_{m+1+i, m+1+i}(B)$ as well as $E^*_{i,i}(B)=E^*_{m+i, m+i}(B)$.

\bigskip
\noindent
\textbf{Case 2:} $E_{i, m+1}$ for $1\le i<m+1$. These basis elements contribute the conditions $\sum_{i=1}^nE^*_{m+1, m+1+i}(B)=0$ as well as $E^*_{m+1, m+i}(B)=0$.

\bigskip
\noindent
 \textbf{Case 3:} $E_{m+1, m+1+i}$ for $1\le i\le n$. These basis elements contribute the conditions $E^*_{1, m+1}(B)=-E^*_{i+1, m+1}(B)$, $E^*_{1, m+1}(B)+E^*_{m+1, m+1}(B)-E^*_{2m+1, 2m+1}(B)=0$ as well as $E^*_{1, m+1}(B)=0$.

\bigskip
\noindent \textbf{Case 4:} $E_{1,1}$. This basis element forces the condition $\sum_{i=1}^nE^*_{1, m+1+i}(B)=0$.

\bigskip
\noindent
 \textbf{Case 5:} $E_{i,i}$ for $1<i\le m+1$. These basis elements force the condition $E^*_{i, m+i}(B)=0$.

\bigskip
\noindent
 \textbf{Case 6:} $E_{i,i}$ for $m+1<i=m+j<2m+1$. These basis elements contribute the conditions $E^*_{1, m+j}(B)+E^*_{j, m+j}(B)=0$.

\bigskip
\noindent
 \textbf{Case 7:} $E_{i, i}$ $2m+1<i=m+j<n+m+1$. This basis element forces the condition $E^*_{1, i}(B)=0$.
 \bigskip

\noindent
We find that Cases 5, 6, 7, and part of Case 3 tell us that $E^*_{1, i}(B)=0$ for $n+1\le i\le m+n+1$ as well as $E^*_{i, m+i}(B)=0$ for $1<i\le m+1$. Case 3 all together gives $E^*_{i, m+1}(B)=0$ for $1\le i\le m$. Case 2 all together gives $\sum_{i=m+1}^nE^*_{m+1, m+1+i}(B)=0$ which contributes $n-m-1$ to $\dim(\ker(B_F))$. Finally, Case 1 and a part of Case 3 allows us to conclude that $E^*_{1,1}(B)=E^*_{i, i}(B)$ for $1<i\le n+m+1$. So, the question remaining is: How many entries of $B$ are unrestricted? These correspond to $E^*_{i, m+j}(B)$ for $1<j\neq i\le n+1$; that is, $(n-1)(m-1)$ entries. Thus, $\dim(\ker(B_F))$ is equal to $$(n-1)(m-1)+n-m-1+1=mn-2m+1=m(n-2)+1,$$ which gives an upper bound on $\ind(\mathfrak{g}(\mathcal{P}(m,1,n)))$. Restricting to $\mathfrak{sl}(n+m+1)$, the result follows. 
\end{proof}

\begin{remark}
The case of $m>n$ follows via a symmetric choice of functional and an argument similar to that used in the proof of Lemma~\ref{lem:ubmn}.
\end{remark}

\subsection{Lower Bound}\label{sec:lb}

In this subsection, we establish lower bounds on the index of $\mathfrak{g}_A(\mathcal{P}(n,1,m))$ which will match the upper bounds found in Section~\ref{sec:ub}. In the case $m=n$, the lower bound follows relatively easily from the structure of $C'(\mathfrak{g}(\mathcal{P}(n,1,n))$. 

\begin{lemma}\label{lem:lbnn}
If $n\in\mathbb{Z}_{>0}$, then $\ind(\mathfrak{g}_A(\mathcal{P}(n,1,n)))\ge n^2-2n+2$.
\end{lemma}
\begin{proof}
Consider the $B_3$ block of $C'(\mathfrak{g}(\mathcal{P}(n,1,n)))$, which is of the form
\[
  \kbordermatrix{
    & \mathbf{E_{1,n+1}}  & \hdots & \mathbf{E_{n,n+1}} & \mathbf{E_{n+1,n+2}} & \hdots & \mathbf{E_{n+1,2n+1}}  \\
    \mathbf{E_{2,n+1}} & 0   & \hdots & 0 & -\frac{E_{2, n+1}}{E_{1, n+1}}E_{1, n+2}+E_{2, n+2} &  \hdots & -\frac{E_{2, n+1}}{E_{1, n+1}}E_{1, 2n+1}+E_{2, 2n+1}     \\ 
    \vdots & \vdots   & \hdots & \vdots & \vdots & \hdots & \vdots    \\
   \mathbf{E_{n,n+1}} & 0   & \hdots & 0 & -\frac{E_{n, n+1}}{E_{1, n+1}}E_{1, n+2}+E_{n, n+2} & \hdots & -\frac{E_{n, n+1}}{E_{1, n+1}}E_{1, 2n+1}+E_{n, 2n+1}    \\ 
   \mathbf{E_{n+1,n+1}} & -E_{1, n+1}   & \hdots & -E_{n, n+1} & E_{n+1, n+2} & \hdots & E_{n+1, 2n+1} \\
   \mathbf{E_{n+1,n+2}} & -E_{1, n+2}  & \hdots & -E_{n, n+2} & \frac{E_{n+1, n+2}}{E_{1, n+1}}E_{1, n+2} & \hdots & \frac{E_{n+1, n+2}}{E_{1, n+1}}E_{1, 2n+1}    \\
   \vdots & \vdots  & \hdots & \vdots & \vdots & \hdots & \vdots     \\
   \mathbf{E_{n+1,2n+1}} & -E_{1, 2n+1}  & \hdots & -E_{n, 2n+1} & \frac{E_{n+1, 2n+1}}{E_{1, n+1}}E_{1, n+2} & \hdots & \frac{E_{n+1, 2n+1}}{E_{1, n+1}}E_{1, 2n+1} \\
  }.
\]
Apply the row operations $\mathbf{E_{n+1,n+i+1}}-\frac{E_{1,n+i+1}}{E_{1,n+1}}\mathbf{E_{n+1,n+1}}$ at row $\mathbf{E_{n+1,n+i+1}}$, for $1\le i\le n$. Now, multiply rows of the form $\mathbf{E_{i,n+1}}$ and $\mathbf{E_{n+1,n+i+1}}$, for $1\le i\le n$, by $E_{1,n+1}$. Consider the sub-matrix $M$ defined by rows and columns $\mathbf{E_{i,n+1}}$ and $\mathbf{E_{n+1,n+i+1}}$, for $1\le i\le n$, which now must be of the form
$$\begin{bmatrix}
     0 & A    \\
     -A^T & B    
\end{bmatrix},$$
\noindent
where 
$$A=\begin{bmatrix}
    E_{1, n+1}E_{2, n+2}-E_{2, n+1}E_{1, n+2} & \hdots & E_{1, n+1}E_{2, 2n+1}-E_{2, n+1}E_{1, 2n+1}     \\ 
    \vdots & \hdots & \vdots    \\
   E_{1, n+1}E_{n, n+2}-E_{n, n+1}E_{1, n+2} & \hdots & E_{1, n+1}E_{n, 2n+1}-E_{n, n+1}E_{1, 2n+1}    
\end{bmatrix},$$ 
\noindent
and
$$B=\scalemath{0.6}{\begin{bmatrix}
    0 & E_{n+1,n+2}E_{1,n+3}-E_{1,n+2}E_{n+1,n+3} & \hdots & E_{n+1,n+2}E_{1,2n}-E_{1,n+2}E_{n+1,2n} & E_{n+1,n+2}E_{1,2n+1}-E_{1,n+1}E_{n+1,2n+1}    \\
    E_{n+1,n+3}E_{1,n+2}-E_{1,n+3}E_{n+1,n+2} & \vdots & \hdots &
    \vdots  & \vdots     \\
    \vdots & \vdots & \hdots & \vdots & \vdots \\
    \vdots & \vdots & \hdots & \vdots & E_{n+1,2n}E_{1,2n+1}-E_{1,2n}E_{n+1,2n+1} \\
    E_{n+1, 2n+1}E_{1, n+2}-E_{1,2n+1}E_{n+1,n+2} & E_{n+1,2n+1}E_{1,n+3}-E_{1,2n+1}E_{n+1,n+3} & \hdots & E_{n+1,2n+1}E_{1,2n}-E_{1,2n+1}E_{n+1,2n} & 0    
\end{bmatrix}};$$ 
that is, $B=-B^T$ and $M$ is skew-symmetric. Note that $M$ has dimension $2n-1$, so that $\det(M)=0$. Thus, there is at most one row of $M$ in the span of the others. Assuming all other rows of $M$ are linearly independent, it then follows from Remark~\ref{rem:contminmax} that the index of $\mathfrak{g}_A(\mathcal{P}(n,1,n))$ is bounded below by $|Rel(\mathcal{P})|-|Ext(\mathcal{P})|+2=n^2-2n+2$.
\end{proof}

The case of $n\neq m$ requires results relating the rank of a skew-symmetric matrix to the matching number of a corresponding graph. First, recall a \textit{matching} of a graph $G=(V,E)$ is a subset $E_M$ of $E$ such that for any two $e_1,e_2\in E_M$ one has $e_1\cap e_2=\emptyset$. Using this idea, the \textit{matching number} of a graph $G$ is the maximum cardinality of a matching on $G$. We will denote the matching number of a graph $G$ by $\nu(G)$. Given a matching $E_M\subset E$ on $G$, a vertex is called \textit{free} if it is not contained in any edge of $E_M$. Now, given a matching on a graph $G$ and two free vertices $v_1$ and $v_2$, a path between $v_1$ and $v_2$ is called \textit{augmenting} if the path alternates between edges contained and not contained in $E_M$, beginning and ending on edges not contained in $E_M$. The following result will be our main tool for computing the matching number of a graph and is standard in the theory of matchings.

\begin{theorem}\label{thm:aug}
Given a graph $G=(V,E)$, a matching $E_M\subset E$ is maximal if and only if there is no augmenting path in $G$ between free vertices of $E_M$.
\end{theorem}

Now, to a given skew-symmetric matrix $M$ one can attach a graph $G^M$ whose adjacency matrix is the same size as $M$ with a 1 in all locations where $m$ is nonzero, and 0's elsewhere. For a poset $\mathcal{P}$, let $G^{\mathcal{P}}$ be the graph attached to the commutator matrix of $\mathfrak{g}_A(\mathcal{P})$. Note in such a graph the vertices will represent basis elements of $\mathfrak{g}_A(\mathcal{P})$ and edges will be defined by pairs of non-commuting basis elements. To get a lower bound on $\ind\mathfrak{g}_A(\mathcal{P}(m, 1, n))$ for $m<n$ we will use the following result which follows from Theorem 2.5 of \textbf{\cite{Matching}}.

\begin{theorem}
The maximum rank of $C(\mathfrak{g}_A(\mathcal{P}))$ is equal to 2$\nu(G^{\mathcal{P}}$).
\end{theorem}

\begin{lemma}\label{lem:lbmn}
If $n,m\in\mathbb{Z}_{>0}$ satisfying $m<n$, then
$\ind\mathfrak{g}_A(\mathcal{P}(m,1,n))=m(n-2)$.
\end{lemma}
\begin{proof}
We will show that the following matching on $G^{\mathcal{P}(m,1,n)}$ is maximal:
\begin{itemize}
    \item $E_{i,i}-E_{m+n+1,m+n+1}$ with $E_{i,n+m+1}$ for $1\le i\le m+1$;
    \item $E_{i,i}-E_{m+n+1,m+n+1}$ with $E_{i-m-1,i}$ for $n+1<i\le 2m+2$;
    \item $E_{i,i}-E_{m+n+1,m+n+1}$ with $E_{m+1,i}$ for $2m+2<i\le n+m+1$;
    \item $E_{i,m+1}$ with $E_{m+1,m+1+i}$ $1\le i\le m$.
\end{itemize}
Note that the only free vertices correspond to basis elements of the form $E_{i,j}$, for $i\in \mathcal{P}$ minimal and $j\in \mathcal{P}$ maximal satisfying $j\neq n+m+1,m+1+i$. To prove that the above matching is maximal, we must show that there does not exist an augmenting path between these free vertices by Theorem~\ref{thm:aug}. Any such path must be of the form illustrated in Figure~\ref{fig:apath}, where to condense notation, for $i\in \mathcal{P}$, we have replaced $E_{i,i}-E_{m+n+1,m+n+1}$ by $E_{i,i}.$
\begin{figure}[H]
$$\begin{tikzpicture}
\node at (-0.5,0.5) {(1) $E_{i,j}$};
\node at (1.5,1.5) {$E_{i,i}$};
\node at (1.5,-0.5) {(2) $E_{j,j}$};
\draw[dashed] (0,1) -- (1,1.5);
\draw[dashed] (0,0) -- (0.5,-0.5);
\draw (2,1.5) -- (3,1.5);
\node at (3.5,1.5) {(1)};
\draw (2,-1) -- (3,-1.5);
\draw (2.5,-0.5) -- (3.5,-0.5);
\node at (3.5,-2) {(1)};
\node at (4.5,-0.5) {(3) $E_{n+1,j}$};
\draw[dashed] (5.5,0) -- (6.5,0.5);
\draw[dashed] (5.5,-1) -- (6.5,-1.5);
\node at (7.5,0.5) {(4) $E_{i,n+1}$};
\node at (7.75,-1.5) {(5) $E_{n+1,n+1}$};
\draw (8.5,0.5) -- (9.5,0.5);
\node at (10.5,0.5) {$E_{n+1,n+1+i}$};
\draw (9,-1.5) -- (10,-1.5);
\node at (11,-1.5) {$E_{n+1,n+m+1}$};
\draw[dashed] (12.5,-1.5) -- (13.5,-1.5);
\node at (14,-1.5) {(4)};
\draw[dashed] (11.5,1) -- (12.5,1.5);
\draw[dashed] (11.5,0) -- (12.5,-0.5);
\draw[dashed] (12,0.5) -- (13,0.5);
\node at (13,2) {(2)};
\node at (13.5,0.5) {(4)};
\node at (13,-0.5) {(5)};
\end{tikzpicture}$$
\caption{Augmenting Path}\label{fig:apath}
\end{figure}
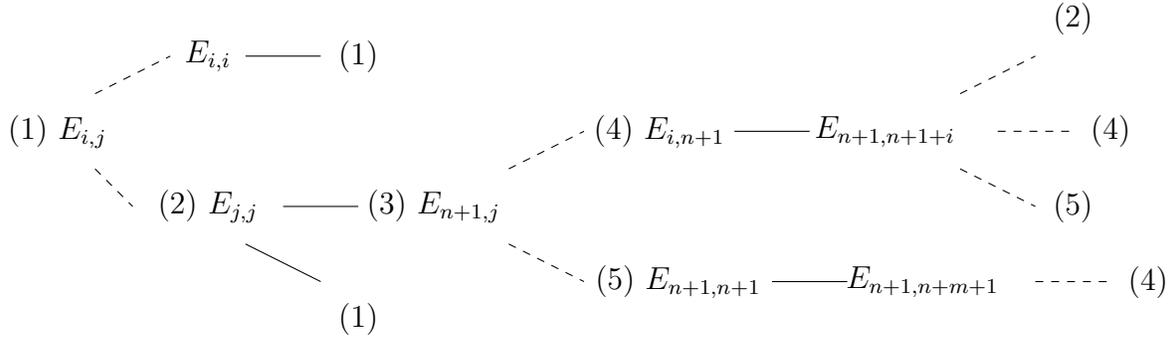
Thus, there can be no augmented path between the free vertices in our matching on $G^{\mathcal{P}(m,1,n)}$, i.e., our matching is maximal. Now the maximum rank is given by double the size of the maximal matching, so that the lower bound on $\ind(\mathfrak{g}_A(\mathcal{P}(m,1,n))$ is given by the number of free vertices, of which there are $m(n-2)$.
\end{proof}

\begin{remark}
The case $n<m$ follows via a symmetric matching and an argument similar to that used in the proof of Lemma~\ref{lem:lbmn}.
\end{remark}

\noindent
Combining Lemmas~\ref{lem:ubnn}$-$\ref{lem:lbmn} establishes Theorem~\ref{thm:CPn1m} of Section~\ref{sec:if}, repeated below for completeness.

\begin{theorem*}
\[\ind\mathfrak{g}(\mathcal{P}(n,1,m)) =  \begin{cases} 
      n^2-2n+2, & n=m; \\
      n(m-2), & m>n; \\
      m(n-2), & n>m.
   \end{cases}
\]
\end{theorem*}

\end{document}